\documentclass{article}

\usepackage[preprint]{neurips_2019}
\usepackage[utf8]{inputenc} 
\usepackage[T1]{fontenc}    
\usepackage{hyperref}       
\usepackage{url}            
\usepackage{booktabs}       
\usepackage{amsfonts}       
\usepackage{nicefrac}       
\usepackage{microtype}      
\usepackage{lipsum}

\usepackage{mathtools}

\usepackage{amsthm, amsmath, amssymb}
\usepackage{graphicx}
\usepackage{pgfplots}
\pgfplotsset{compat=1.9}
\usepackage{float}
\usepackage{enumerate}

\usepackage{tikz}
\usetikzlibrary{patterns}

\newtheorem{theorem}{Theorem}[section]

\newtheorem{corollary}{Corollary}[section]
\newtheorem{lemma}{Lemma}[section]

\DeclarePairedDelimiterX{\inp}[2]{\langle}{\rangle}{#1, #2}

\newcommand{\ip}[2]{\left\langle #1 , #2 \right\rangle}

\newcommand{\GLG}[1]{\mathrm{GL}_{#1}}
\newcommand{\OG}[1]{\mathcal{O}_{#1}}
\newcommand{\SM}[2]{\mathrm{St}_{#1,#2}}

\newcommand{\SP}{{\cal X}}

\newcommand{\cM}{{\cal M}}
\newcommand{\cN}{{\cal N}}

\newcommand{\cR}{{\cal R}}
\newcommand{\cS}{{\cal S}}
\newcommand{\cT}{{\cal T}}

\newcommand{\cV}{{\cal V}}

\newcommand{\cX}{{\cal X}}

\newcommand{\R}{\mathbb R}

\newcommand{\Int}[2]{{[#1 \!:\! #2]}}	
\newcommand{\trace}{\mathop{\mathrm{trace}}}

\newcommand{\vspan}{\mathop{\mathrm{span}}}
\newcommand{\rank}{\mathop{\mathrm{rank}}}

\newcommand{\ra}{\rightarrow}

\newcommand{\by}{\times}
\newcommand{\shalf}{{\textstyle\frac{1}{2}} }

\newcommand{\nhalf}{\nicefrac{1}{2}}
\newcommand{\vect}[1]{\mathbf{#1}}
\newcommand{\vzero}{\vect{0}}

\usepackage[utf8]{inputenc} 
\usepackage[T1]{fontenc}    
\usepackage{hyperref}       
\usepackage{url}            
\usepackage{booktabs}       
\usepackage{amsfonts}       
\usepackage{nicefrac}       
\usepackage{microtype}      

\newcommand{\nomiss}{maximal}
\newcommand{\missatp}{not maximal}

\newcommand{\ve}{{\mathbf e}}
\newcommand{\WS}{{\cal X}}
\newcommand{\GLGorb}{\Theta}
\newcommand{\OGorb}{{\mathcal O}}
\newcommand{\vertt}[1]{{\left\vert\kern-0.25ex\left\vert\kern-0.25ex\left\vert #1 
    \right\vert\kern-0.25ex\right\vert\kern-0.25ex\right\vert}}

\newcommand{\GC}{\tilde G}
\newcommand{\HC}{\tilde H}

\title{The Landscape of Matrix Factorization Revisited}

\author{%
  Hossein~Valavi , Sulin Liu, and Peter J. Ramadge\\
  Department of Electrical Engineering, Princeton University\\
  Princeton, NJ 08544 \\
  \texttt{hvalavi, liu, ramadge @princeton.edu}
}

\begin{document}

\maketitle

\begin{abstract}
We revisit the landscape of the simple matrix factorization problem. 
For low-rank matrix factorization, prior work has shown that there exist infinitely many critical points all of which are either global minima or strict saddles. At a strict saddle the minimum eigenvalue of the Hessian is negative. Of interest is whether this minimum eigenvalue 
is uniformly bounded below zero over all strict saddles.
To answer this we consider orbits of critical points under 
the general linear group. For each orbit we identify a representative point, called a canonical point. If a canonical point is a strict saddle, 
so is every point on its orbit. We derive an expression for the minimum eigenvalue of the Hessian at each canonical strict saddle and use this to show that the minimum eigenvalue of the Hessian over the set of strict saddles is not uniformly bounded below zero. 
We also show that a known invariance property of gradient flow ensures the solution of gradient flow only encounters critical points on an invariant manifold $\cM_C$ determined by the initial condition.
We show that, in contrast to the general situation, the minimum eigenvalue of strict saddles in $\cM_\vzero$  is uniformly 
bounded below zero. We obtain an expression for this bound 
in terms of the singular values of the matrix being factorized.
This bound depends on the size of the nonzero singular values and on the separation between distinct nonzero singular values of the matrix.
\end{abstract}

\section{Introduction} \label{sec:intro}

Factor analysis  is a well known problem in machine learning and many methods have been introduced for posing and solving problems of this form. Some of the best known examples include 
Principal Component Analysis (PCA) \cite{hotelling1933analysis,hotelling1936,jolliffe2011principal}, 
Canonical Correlation Analysis (CCA) \cite{hardoon2004canonical}, 
Independent Component Analysis (ICA) \cite{hyvarinen2000independent},
Positive Matrix Factorization (PMF) \cite{paatero1994positive}, 
and Dictionary Learning~\cite{mairal2009online}.
All of these problems share the characteristic of being non-convex 
optimization problems over matrix arguments. 
While some of the problems can be solved using well established tools (e.g. SVD), 
others require the application of iterative solution methods such as gradient descent.
This motivates developing a better understanding the landscape of these problems. This has gained additional momentum recently because of a connection to representation learning in deep networks. 

In the late 1980's, Baldi and Hornik~\cite{baldi1989neural} examined the landscape question in the context of training a one hidden layer low-rank linear neural network (a linear auto-encoder). Their paper provides a characterization of the associated critical points, proves that these are either global minima or (strict) saddle points, and  shows that the global minimum value corresponds to the residual of the projection of the training data onto the subspace generated by the first principal vectors of a covariance matrix associated with the training patterns. 
This connection to PCA had been previously established by Bourlard in  \cite{Bourlard1988}. The results in \cite{baldi1989neural} are proved by vectorizing the relevant matrix differentials and seeking small perturbations around a critical point to make the objective smaller. 

Recently, there has been a surge of interest in characterizing the global landscape of the objective functions used in these types of problems. 
Chen et al.~\cite{chen2018landscape} study the landscape of 
the generalized eigenvalue problem and characterize the landscape by looking at the Hessian of the Lagrangian function by vectorizing the relevant differentials.
Li et al.~\cite{li2019symmetry} examine the landscape problem under the lens of invariant groups. In particular, using the properties of a rotational symmetry group, they  show that the parameter space can be divided into three regions. The first contains all strict saddle points, where the objective has a negative curvature, the second contains neighborhoods of all global minima, and the third is the complement of the first two regions. 
Mohammadi et al. \cite{mohammadi2018stability} examine the equilibrium points of the best rank-one approximation under the gradient flow. 
Ge et al. \cite{ge2017no} study the landscape of the matrix sensing, matrix completion and robust PCA problems, showing that for these problems all local minima are global minima.
Sun et al.~\cite{sun2018geometric,sun2015complete} study the phase retrieval problems and dictionary recovery.  
Boumal \cite{boumal2016nonconvex} studies phase synchronization. 
Much effort has also recently been devoted to studying the related problem of the optimization landscape of deep neural networks under simplified assumptions~
\cite{kawaguchi2016deep,nguyen2017loss,hardt2016identity,ge2017learning}.

Several papers have investigated provable guarantees for the local and global convergence properties of optimization on non-convex problems under gradient flow and gradient descent. A common theme in these papers is the concept of a strict saddle function introduced in \cite{ge2015escaping}. 
The main required properties are that the functions are twice-differentiable and that the minimum eigenvalue of the Hessian is positive at all local minima and negative at all other critical points. For such functions, with high probability, stochastic gradient descent (SGD) converges to a local minimum in a polynomial number of iterations.  
Several subsequent papers show that a wide class of non-convex functions are indeed strict saddle functions. 
For instance, orthogonal tensor decomporition \cite{ge2015escaping}, deep linear neural networks \cite{kawaguchi2016deep}, 
deep linear residual neural networks \cite{hardt2016identity}, 
matrix completeion \cite{ge2016matrix}, 
generalized phase retrieval problem \cite{sun2018geometric}, 
complete dictionary recovery over the sphere \cite{sun2015complete}, low-rank matrix recovery \cite{bhojanapalli2016global}, 
are all examples of problems that satisfy the strict saddle property. 
Lee et al. \cite{lee2016gradient} use the stable manifold theorem to show that for a twice continuously differentiable function with the strict saddle property, almost surely, gradient descent with random initialization converges to a local minimum or negative infinity.
Jin et al.~\cite{jin2017escape} show that for $l$-smooth and $\rho$-Hessian Lipschitz functions with the strict saddle property, with high probability, perturbed gradient descent (PGD) converges to a local minimum in a poly-logarithmic number of iterations.

Other papers have investigated the implicit constraints imposed by gradient flow in training over-parameterized models such as deep neural networks \cite{AroraICML2018a, DuNIPS2018}. Arora et al.~\cite{AroraICML2018a} considers over-parameterized multi-layer linear neural networks and shows that gradient flow implicitly balances the underlying factors. Du et al. \cite{DuNIPS2018} extends this result to fully-connected and convolutional linear sections of multi-layer neural networks.

This paper revisits the landscape of the simple matrix factorization (MF) problem. Prior results indicate that critical points are either global minima or strict saddles~\cite{baldi1989neural}. Our analysis is based on orbits of factorizations under the general linear group. After defining these orbits, we construct a representative point (called a canonical point) on each orbit of critical points. This construction exploits a known connection to PCA/SVD. We then obtain an expression for the minimum eigenvalue of the Hessian at each canonical strict saddle. 
This bound can be moved along a curve in the orbit to show that there is no uniform negative bound on the minimum eigenvalue of the 
Hessian over the orbit. This implies that the minimum eigenvalue of the Hessian is not uniformly bounded below zero over all strict saddles.  
We then analyze how an invariance property of gradient flow impacts the strict saddles that can be encountered. This natural invariant depends on the initial condition for gradient flow. We focus on one interesting form of the invariant. When the initial conditions start in this invariant manifold, the flow restricts the symmetry of the problem to the orthogonal group and this allows us to to give a uniform negative bound on the minimum eigenvalue of the Hessian map at all strict saddles on the invariant manifold. Moreover, we provide exact expressions for the minimum eigenvalue of the Hessian at each strict saddle.
Our results are applicable to the situations $k< r$ and $k>r,$ 
where $k$ denotes the rank of the factorization.

\section{Preliminaries} \label{sec:prel}
For positive integers $m, n$, let $\R^{m\times n}$ denote the set of $m\times n$ real matrices,
$\GLG n \subset \R^{n\times n}$ denote the general linear group of invertible $n\times n$ matrices, 
$\OG n \subset \GLG n$ denote the group of orthogonal $n \times n$ matrices, and for $k \leq m,$ $\SM mk$ denote the subset of $m\times k$ real matrices with orthonormal columns.
For $A, B\in \R^{m\times n}$, $A_{:,k}$ (resp. $A_{k,:}$) denotes the $k$-th column (resp. row) of $A,$ 
$\ip A B$ denotes the standard inner product of $A$ and $B,$ 
and $\| A \|_{F}$ denotes the Frobenius norm of $A.$ 

Let $f \colon \R^{m\by n} \to \R$ be a twice continuously 
differentiable function. The derivative of $f$ with respect to $X$ evaluated at a given point $X_0$ is a linear map from 
$\R^{m\by n}$ to $\R.$ Its action on $H\in \R^{m\by n}$, 
denoted by $D_X f(X_0) [H],$  satisfies
$$
D_X f(X_0) [H] = \lim_{\alpha\ra 0} \frac{f(X_0)+\alpha H) - f(X_0)}{\alpha} .
$$  
The gradient of $f$ at $X_0$, denoted $\nabla_X f(X_0),$ 
is the unique point in $\R^{m\by n}$ such that
$$
D_X f(X_0) [H] = \ip{\nabla_X f(X_0)}{H}. 
$$   
When no confusion is possible we simply write $D f(X_0)$ 
and $\nabla f(X_0).$
The gradient can also be regarded as a function 
$X \mapsto \nabla f(X).$ The derivative of this function 
at $X_0$ is a linear map 
$\nabla^2_X f(X_0)\colon \R^{m\by n} \to \R^{m\by n}.$ 
Its action on $H\in \R^{m\by n}$ is denoted by
$
\nabla^2_X f(X_0)[H].
$
When no confusion is possible we simply write $\nabla^2 f(X_0).$
It is natural to call the linear map $\nabla^2 f(X_0)$ the Hessian map.
The second derivative of $f$ is then defined by 
\begin{equation}\label{eq:D2def}
D^2_X f(X_0)[H,K] = \ip{\nabla^2_Xf(X_0)[K]}{H}. 
\end{equation}
We will always evaluate the second derivative with $K=H$. In this case, we simply write $D^2f(X_0)[H]$. The second derivative is then a scalar valued function from $\R^{m\by n}$ to $\R$. The linear function 
$\nabla^2 f(X_0)[H]$ embedded in the second derivative brings in eigenvalues and eigenvectors associated with the second derivative. 

$X_0$ is a critical point of $f$ if $\nabla J(X_0)=\vzero$. 
The second order Taylor series of $f$ about a critical point $X_0$ 
in the direction of $H$ is
\begin{align}
f(X_0+tH) 
&= f(X_0) +\nhalf  t^2 D^2f(X_0)[H] .\label{eq:sots2}
\end{align}
Together with eigenvalues of $\nabla^2 f(X_0),$ this
can be used to partially classify critical points. 
However, this fails if there is nonzero $H$ with $D^2f(X_0)[H] =0.$
In this event, $X_0$ is termed a degenerate critical point. 
Motivated by the observation that many non-convex optimization problems have degenerate critical points, the concepts of a strict saddle and a strict saddle function have been introduced \cite{ge2015escaping}.  A strict saddle is a critical point for which $\nabla^2f(X_0)$ has a least one negative eigenvalue (this includes local maxima). 

\subsection{Basic Matrix Factorization}
In unconstrained matrix factorization, given $X\in \R^{m \times n}$ 
with $\rank(X)=r$, and a factorization dimension $k$, 
we seek a solution of 
\begin{equation} \label{eq:J}
\min_{(W, S) \in \WS }
    J(W,S) = \nhalf \,  \| X - W S \|_{F}^{2}.
\end{equation}
Here $\WS \triangleq \R^{m\by k} \times \R^{k\by n}$ with
inner product $\ip{(G,H)}{(G^\prime, H^\prime)} =
\ip{G}{G^\prime} + \ip{H}{H^\prime}$ and associated norm 
$\|(G,H)\|_F^2 = \|G\|_F^2 + \|H\|_F^2$.   

When $k\leq r$, problem \eqref{eq:J} has a well known solution derived 
from any compact SVD $X=U\Sigma V^T$.
Let $U_k$ (resp. $V_k$) consist of the first $k$ columns of 
$U$ (resp. $V$) and  $\Sigma_k$ be the top left $k\by k$ 
submatrix of $\Sigma$. 
Then $(W,S)=(U_k\sqrt{\Sigma_k}, \sqrt{\Sigma_k} V_k^T)$ is a 
global minimum of \eqref{eq:J}. However, rather than solving \eqref{eq:J} using the SVD, we are interested in finding a solution using gradient descent methods.  This motivates an interest in the 
global landscape of $J$.

We will work extensively with the gradient, Hessian, 
and second derivative of $J$ defined in \eqref{eq:J}. 
Setting $E \triangleq WS-X,$ we list the easily obtained equations 
for these functions below.
\begin{align}
\nabla J(W,S) 
&=  \begin{pmatrix} E S^T, ~ W^T E \end{pmatrix}\label{eq:gradJ}\\
\nabla^2 J(W,S)[(G,H)] &= (GSS^T + WHS^T +EH^T~ ,~ W^TWH + W^TGS +G^TE ) \label{eq:hessianmap},\\
D^2 J(W,S) [(G,H)] 
&= \|GS\|_F^2  + \|WH\|_F^2 + 2 \trace(  H^TW^TGS+  H^TG^TE) \label{eq:D2J},
\end{align}

A point $(W,S)$ is a critical point of $J$ if $\nabla J(W,S)=\vzero$. 
By \eqref{eq:gradJ} this is equivalent to 
\begin{align}
\label{eq:stpoint}
ES^T=(WS  - X)S^T =\vzero
\quad\textrm{and} \quad
W^T E= W^T(WS  - X) =\vzero.
\end{align}
 
Because $J$ has a continuous symmetry group that leaves its value invariant (see \S\ref{sec:mf-single}), every nonzero critical point of 
$J$ is degenerate.
To see this, let $(W,S)\neq \vzero$ be a critical point, 
and $K\in \R^{k\by k}$ be any matrix with $(WK, -KS)$ nonzero.
Then using \eqref{eq:hessianmap} and \eqref{eq:stpoint} we see that
\begin{align*}
\nabla^2 J(W,S) [(WK, -KS)]
= ( - ES^TK, K^TW^TE)
= \vzero.
\end{align*}
Thus $(WK, -KS)$ is an eigenvector of $\nabla^2 J(W,S)$ with eigenvalue $0$.  
To address this degeneracy, we employ the concept of a strict saddle \cite{ge2015escaping}. 
A critical point $(W,S)$ is a strict saddle if
$\lambda_{\min}(\nabla^2J(W,S) ) <0$. The corresponding eigenvector provides an escape  direction from the neighborhood of the critical point. 
Of particular interest is determining a negative upper bound on $\lambda_{\min}(\nabla^2J(W,S))$ for the strict saddles of $J$.

\subsubsection{Orbits and Basic Orbit Properties}
For $A\in \GLG k$, let $L_A\colon \cX\to \cX$ denote the linear map
$L_A\colon (G,H)\mapsto (GA, A^{-1}H).$
Then for given $(W,S)\in \WS$ let 
\begin{align}
\GLGorb(W,S) \triangleq \{ L_A(W,S) \colon A\in \GLG k\} \label{eq:orbit}
\qquad\textrm{and} \qquad
\OGorb(W,S) \triangleq \{ L_Q(W,S) \colon Q\in \OG k\} .
\end{align}
We call $\GLGorb(W,S)$  the orbit of $(W,S)$ under $\GLG k$, and
$\OGorb(W,S)$) the suborbit of $(W,S)$ under $\OG k$.
The value $J(W,S)$ is constant on each orbit. 
Hence if any point on an orbit is a global minimum, 
all points on the orbit are global minima.
Moreover, the gradient  $\nabla J(W,S)$ must be orthogonal to 
$\Theta(W,S)$ at $(W,S)$. 
To prove this, take the derivative of $(WA,A^{-1}S)$ 
with respect to $A$, and then set $A=I_k$. 
This yields the set of tangents to $\Theta(W,S)$ at $(W,S)$:
\begin{equation}\label{eq:glg_tang}
T_{W,S} = \{ (WK, -KS)\colon  K \in \R^{k \times k} \}.
\end{equation}
Orthogonality is then verified by taking the inner product of any element in \eqref{eq:glg_tang} with \eqref{eq:gradJ}.

Qualitative properties of point neighborhoods are 
also ``preserved'' along the orbit. 
To see this let $(G,H)\in \SP$ and $t>0$.  
Then for any $A \in \GLG k$,
$$
J(W+tG, S+tH) = J(WA + t(GA), A^{-1}S + t( A^{-1}H) ).
$$
So the values of $J$ traced out moving from $(W,S)$ in a line in the direction of $(G,H)$ are the same as those traced out moving 
from $(WA, A^{-1}S)$ along a line in the direction of $(GA, A^{-1}H)$. 
The following lemma (proved in \S\ref{sec:auxlem}) transfers these observations to the derivatives of interest.

\begin{lemma}\label{lem:LAderiv}
For all $(W,S)\in \WS$ and $A \in \GLG k$: 

(a) $\nabla J(L_A(W,S) ) = L_{A^{-T}} \nabla J(W,S)$, 
	
(b) $DJ(L_A(W,S)) [(G,H)] = D J(W,S)[L_{A^{-1}}(G,H)] $,
	
(c) $\nabla^2 J(L_A(W,S))[(G,H)] 
		= L_{A^{-T}}( \nabla^2 J(W,S)[L_{A^{-1}}(G,H)] )$,
		
(d) $D^2 J(L_A(W,S)) [(G, H)] = D^2 J(W,S) [L_{A^{-1}}(G,H)] .$
\end{lemma}

For a critical point $(W,S),$ the inertia of $\nabla^2 J(W,S)$ 
is the triple $(i_+, i_-, i_0),$ consisting of the number of its positive, negative, and zero eigenvalues, respectively.
We have already noted that if any point on an orbit is a global minimum, 
so are all points on the orbit. The same holds for strict saddles.
In addition, the inertia of $\nabla^2 J$ is an invariant of an orbit.
The following theorem (proved \S\ref{sec:auxlem}) lists these 
and other orbit properties.

\begin{theorem}\label{thm:orbits}
Let $(W,S)\in \WS$ and $A\in \GLG k$. Then 

(a) If $(W,S)$ is a critical point (global minimum, strict saddle), 
so are all points in $\GLGorb(W,S)$.
	
(b) $\nabla^2 J$ has the same inertia at all points in $\GLGorb(W,S).$ 

(c) $\nabla^2 J$ has the same eigenvalues at all points in $\OGorb(W,S).$ 

(d) For each $A\in \GLG k$, 
$
\textstyle \lambda_{\min} (\nabla^2 J(L_A(W, S) ) )
\leq 
\frac{\lambda_{\min} (\nabla^2 J(W,S) )}{\max\{ \lambda_{\max}(AA^T),	
	\lambda_{\min}^{-1} (AA^T)\} }.
$

\end{theorem}

Note that 
$\max\{ \lambda_{\max}(AA^T), \lambda_{\min}^{-1} (AA^T)\}
	\geq 1$  with equality when $A \in \OG k.$ 
For an orthogonal matrix $A,$ 
$\lambda_{\min}(\nabla^2 J(L_A(W,S)) = \lambda_{\min}(\nabla^2 J(W,S)).$ 
In this case, the bound in (d) is an equality. 

In summary, if an orbit contains a critical point, 
all points in the orbit are critical points; 
if it contains a strict saddle, all points in the orbit are strict saddles, 
and if it contains a global minimum, all points on the orbit 
are global minima. 
In addition, the number of positive, negative and zero eigenvalues 
of $\nabla^2 J$ are invariants of each orbit. 
Using the known result  that all critical points are 
either global minima and strict saddles \cite{baldi1989neural}, 
we see that there are three kinds of obits: 
orbits of global minima, orbits of strict saddles, 
and orbits of noncritical points.

\section{The Landscape of Matrix Factorization} \label{sec:mf-single} 

Our approach is to first construct a representative point, 
called a canonical point, on each orbit of critical points. 
We will use each canonical point to reason about all of 
the points on its orbit. In particular, we examine 
$\lambda_{\min}(\nabla^2 J(W,S))$ over orbits of strict saddles. 
As a by-product we recover the known result that every critical 
point of $J$ is either a global minimum or a strict saddle \cite{baldi1989neural}.

The construction below exploits the known connection 
between matrix factorization and an SVD of the data matrix $X$. 
We will use an SVD of $X$ as a theoretical tool in our 
definitions and proofs. However, we do not intend to 
numerically compute an SVD of $X.$ For clarity we assume $m < n.$ 
This is without loss of generality since symmetric arguments 
apply when $n<m.$
Since $m<n,$ the $m\by m$ matrix $XX^T$ is of interest.
This has $r$ positive eigenvalues and $m-r$ zero eigenvalues. 
Denote these by 
$\sigma_1^2 \geq \sigma_2^2 \geq \cdots \geq \sigma_r^2 >0$ and 
$\sigma^2_{r+1}= \cdots = \sigma_m^2=0. $
Let $u_1,\dots, u_m$ denote a set of $m$ corresponding orthonormal eigenvectors 
with $XX^T u_i = \sigma_j^2 u_i$ for $i\in \Int 1r$, and $XX^T u_i =\vzero$  
for $i\in \Int {r+1}m$.  These orthonormal eigenvectors may not be unique 
(the nonzero eigenvalues may be repeated).
Place the eigenvectors in the columns of $U\in \R^{m\by m}$
and form a compatible full SVD 
\begin{equation}\label{eq:fullsvdX}
\textstyle
X=U\Sigma V^T =\sum_{i=1}^m \sigma_i u_i v_i^T.
\end{equation}
The singular values of $X$ are unique, but in general $U$ and $V$ are not.
It does not matter for our purposes which SVD of $X$ is selected as long as it is used consistently.

\subsection{Orbit Representation: Canonical Points} 

Fix a factorization dimension $k,$ and for $q \in \Int 1{\min\{k,m\}},$
place $q$ selected singular values of $X$, 
denoted by  $\lambda^2_1 \geq \lambda^2_2 \geq \cdots \geq \lambda^2_q$, 
in decreasing order in a diagonal matrix $\Lambda^2 \in \R^{q\by q},$ 
and a corresponding set of $q$ left singular vectors of $X$ 
in the columns of $\bar U \in \R^{m\times q}.$
So $XX^T \bar U = \bar U \Lambda^2$. 
Then $\bar U^T X = \bar U^T U \Sigma V^T = \Lambda \bar V^T,$ 
where $\bar V$ denotes the submatrix of $V$ in \eqref{eq:fullsvdX} 
corresponding to the columns of $\bar U.$ 
Let $V_0 =[ v_{r+1},\dots, v_n]$ and $C_0 \in \R^{(n-r)\by (k-q)}$. 
Now form 
\begin{equation} \label{eq:bsp}
(W_c, S_c)= 
\left ( \begin{bmatrix} \bar U  & \vzero_{m\times(k-q)}\end{bmatrix}, 
\begin{bmatrix} \Lambda {\bar V}^T\\ C_0^T V_0^T \end{bmatrix} \right).
\end{equation}
Any point of the form \eqref{eq:bsp} will be called a canonical point.

\begin{theorem}\label{thm:bsp}
For $1\leq q \leq \min\{k,m\}$, the following hold:

(a) $(W_c, S_c)$ in \eqref{eq:bsp} is a critical point of $J$ 
with 
$J (W_c, S_c) = \nicefrac{1}{2} \left( \sum_{i=1}^r \sigma_i^2 
	- \sum_{j=1}^q \lambda^2_j \right).$
	 
(b) If $(W,S)$ is a critical point of $J$ with $\rank(W)=q$, 
then there exists $(W_c, S_c)$ of the form \eqref{eq:bsp} 
and $A\in \GLG k$ such that $(W, S) =  L_A( W_c , S_c)$. 
\end{theorem}

\begin{proof}
(a) 
We show that $(W_c, S_c)$ satisfies \eqref{eq:stpoint}.
$(W_cS_c-X)S_c^T 
= (\bar U \Lambda {\bar V}^T -U\Sigma V^T )[  \bar V  \Lambda , V_0 C_0] 
=\vzero,$ and
$W_c^T(W_c S_c - X)
= \bigl[ \begin{smallmatrix} {\bar U}^T \\ \vzero_{(k-q)\times m} \end{smallmatrix} \bigr] 
( \bar U \Lambda {\bar V}^T   - U\Sigma V^T )
= \vzero.$
In addition,
$J(W_c, S_c) = \nicefrac{1}{2} \| \bar U \Lambda {\bar V}^T -U\Sigma V^T \|_F^2 = \nicefrac{1}{2} \left( \sum_{i=1}^r \sigma_i^2 - \sum_{j=1}^q \lambda^2_j \right).$\\
(b) Step (i).
Since $\rank(W) =q,$ there is a permutation matrix $P\in \GLG k$ such that the first $q$ columns of $WP$ are linearly independent. 
If $\hat W$ denotes the matrix of these first $q$ columns, then
\begin{equation}\label{eq:step1}
W = \begin{bmatrix} \hat W & \vzero_{m\times (k-q)}\end{bmatrix}] 
\begin{bmatrix} I_q & F \\ \vzero_{(k-q)\times q} & I_{(k-q)} 
	\end{bmatrix} P^T,
\end{equation}
where $F$ is determined by the last $k-q$ columns of $WP$. \\
Step (ii). Let $\hat U \hat \Sigma {\hat V}^T$ be a compact SVD of $\hat W$. 
Noting that $\hat U\in \SM mq$ and $\hat \Sigma, \hat V  \in \GLG q,$ 
we have
\begin{equation} \label{eq:step2}
W = \begin{bmatrix} \hat U  & \vzero_{m\times (k-q)}\end{bmatrix} \begin{bmatrix} \hat \Sigma \hat V^T & \hat \Sigma\hat V^T F \\ \vzero_{(k-q)\times q} & I_{(k-q)} \end{bmatrix} P^T.
\end{equation}
Let $C$ denote the product of the two rightmost matrices in \eqref{eq:step2}.
Since $\hat \Sigma, \hat V  \in \GLG q$, $C\in \GLG k$. Let 
$(W_1, S_1) = L_{C^{-1}} (W, S).$ 
By \eqref{eq:step2}, 
$W_1 = \begin{bmatrix} \hat U & \vzero_{m\times (k-q)}\end{bmatrix},$ 
and by Theorem \ref{thm:orbits}, $(W_1, S_1)$ is a critical point.
Write $S_1^T = \begin{bmatrix} S_a^T & S_b^T \end{bmatrix}$ with $S_a \in \R^{q\by n}$
and $S_b\in \R^{(k-q)\by n}.$ 
Then by \eqref{eq:stpoint},
$$
(\hat U S_a -X) \begin{bmatrix} S_a^T & S_b^T \end{bmatrix} =\vzero
\quad \textrm{and}\quad 
\begin{bmatrix} {\hat U}^T\\ \vzero_{{(k-q)}\times m} \end{bmatrix}  [\hat U   S_a - X]  =\vzero.
$$
The second condition implies  $S_a =  {\hat U}^T X$. Then the first implies
(c-i)
$(\hat U {\hat U}^T - I) XX^T \hat U  =\vzero$ 
and  
(c-ii)
$(\hat U \hat U^T -I )X S_b^T =\vzero.$\\
Step (iii). Condition (c-i) implies that range of $\hat U$ is invariant under $XX^T,$ 
i.e., $XX^T \cR(\hat U) \subseteq \cR(\hat U)$. Since $\cR(\hat U)$ has dimension $q$, there are $q$ orthonormal eigenvectors of $XX^T$ that form a basis for $\cR(\hat U)$. 
Let $\bar U$ be the matrix with this basis as its columns arranged in decreasing order of the corresponding eigenvalue. 
Since every column of $\hat U$ has a representation in this basis,
for some $Q\in \GLG q$, $\hat U = \bar U Q$. 
In addition, $I_k = {\hat U}^T \hat U = Q^T Q$. 
Hence $Q \in \OG k$.
Now rewrite \eqref{eq:step2} as
\begin{equation} \label{eq:step3}
W = \begin{bmatrix} \bar U & \vzero_{m\times (k-q)}\end{bmatrix}
 \begin{bmatrix} Q & \vzero_{} \\ \vzero_{(k-q)\times q} & I_{(k-q)} \end{bmatrix} 
 \begin{bmatrix} \hat \Sigma\hat V^T & \hat \Sigma\hat V^T F \\ \vzero_{(k-q)\times q} & I_{(k-q)} \end{bmatrix} P^T,
\end{equation}
Let $D$ denote the matrix in \eqref{eq:step3} containing $Q$, and let 
$(W_2, S_2) = L_{(DC)^{-1}}(W,S)$. 
Then $(W_2, S_2)$ is a critical point with
$W_2 =W(DC)^{-1}  = \begin{bmatrix} \bar U & \vzero_{m\times (k-q)} \end{bmatrix}$,  
and
$S_2=(DC) S= D \bigl[\begin{smallmatrix} S_a\\ S_b\end{smallmatrix} \bigr]  
= \bigl[ \begin{smallmatrix} Q{\hat U}^T X\\ S_b\end{smallmatrix} \bigr]
=  \bigl[ \begin{smallmatrix} {\bar U}^T X\\ S_b\end{smallmatrix} \bigr]
=  \bigl[ \begin{smallmatrix} \Lambda {\bar V}^T \\ S_b\end{smallmatrix} \bigr]. $
By construction, 
$\Lambda^2$ a diagonal matrix with eigenvalues of $XX^T$ arranged in decreasing order down its diagonal, the columns of $\bar U$ are corresponding eigenvectors of $XX^T$ with $XX^T\bar U = \bar U  \Lambda^2$, and by  (c-ii) $(\bar U \bar U^T -I )X {S}_b^T =\vzero$.\\
Step (iv).
Lemma \ref{lem:S2props} shows that $S_b^T = \bar V \bar C + V_0 C_0$ where $V_0 =[ v_{r+1},\dots, v_n]$,  $\bar C \in \R^{q\by (k-q)}$, and $C_0 \in \R^{(n-r)\by (k-q)}$. Hence 
\begin{equation} \label{eq:bsp0}
(W_2, S_2)= 
\left ( \begin{bmatrix} \bar U  & \vzero_{m\times(k-q)}\end{bmatrix}, 
\begin{bmatrix} \Lambda {\bar V}^T\\ \bar C^T \bar V^T + C_0^T V_0^T \end{bmatrix} \right).
\end{equation}
Step (v). $(W_c, S_c) = \left(
\bigl [ \begin{smallmatrix} \bar U  & \vzero_{m\times(k-q)}\end{smallmatrix} \bigr ], 
\bigl [ \begin{smallmatrix} \Lambda {\bar V}^T\\ C_0^T V_0^T \end{smallmatrix} \bigr ]
 \right)$ 
has the canonical form \eqref{eq:bsp}. 
Moreover,  if
$
E = \bigl[\begin{smallmatrix} I_q & \vzero_{(k-q)\by q}\\ -\bar C^T\Lambda^{-1} & I_{k-q}\end{smallmatrix}\bigr] \in \GLG q$
with
$E^{-1} = \bigl[\begin{smallmatrix} I_q & \vzero_{(k-q)\by q}\\ -\bar C^T\Lambda^{-1} & I_{k-q}\end{smallmatrix}\bigr],
$ 
then
$\bigl[\begin{smallmatrix} \bar U & \vzero \end{smallmatrix}\bigr] E
	= \bigl[\begin{smallmatrix} \bar U & \vzero \end{smallmatrix}\bigr]$ 
and 
$E^{-1} \bigl[ 
	\begin{smallmatrix} \Lambda \bar V^T \\ C_0^TV_0^T 
	\end{smallmatrix} \bigr]
	= 
	\bigl[ \begin{smallmatrix} 
	\Lambda \bar V^T \\ \bar C^T\bar V^T +C_0^TV_0^T\end{smallmatrix}\bigr].
$
Hence $L_E(W_c,S_c) = (W_2, S_2)=L_{(DC)^{-1}}(W,S).$ 
Thus $(W,S) = L_{DCE}(W_c,S_c).$
\end{proof}

\begin{figure}[t]
\centering
\begin{tikzpicture}
     \draw[ smooth, ultra thick, double   distance=20pt] plot coordinates{(-5,2) (-3,1.5) (-1,0) (1,0) (3,1.5) (5,2)};
     \draw [fill=white,white] (5.2,1.64) rectangle (5,1.4);     
     \draw [fill=white,white] (-5.2,1.64) rectangle (-5,1.4);           
     \draw [dotted, thick] (-2.6,1.25) ellipse (0.2cm and 0.4cm);
     \draw [dotted, thick] (2.8,1.4) ellipse (0.2cm and 0.4cm);
    \draw [-., ultra thick] (4.93,2) ellipse (0.1cm and 0.4cm);
    \draw [-., ultra thick] (-4.93,2) ellipse (0.1cm and 0.4cm);
    \draw[->,ultra thick] (5.1,2) -- (6,2.1);
    \draw[->,ultra thick] (-5.1,2) -- (-6,2.1);    
     \filldraw [black] (-2.8,1.4) circle (2pt);
     \draw  (-5,1) node{\small $(W,S) \in \R^{m \times k} \times \R^{k \times n}$};
     \filldraw [blue] (2.95,1.2) circle (2pt);
     \draw [black] (4,1) node{\small $(\hat U, \hat \Sigma \hat V^T S)$};  
     \filldraw [red] (2.65,1.6) circle (2pt);
     \draw [black] (1.9,2) node{\small $(\bar U, \Lambda \bar V^T)$};  
     \filldraw [gray] (-2.45,1) circle (2pt);
     \draw [black] (-4,.4) node{\small $L_{V}(W,S) = (\hat U\hat \Sigma, \hat V^T S)$};       
\end{tikzpicture}
\caption{\small A conceptual view of the orbit $\Theta(W,S)$. $(W,S)$ is a critical point with $\rank(W)=\min\{k,m\}$ and compact SVD $W = \hat U \hat \Sigma \hat V^T$. 
$L_{\hat V}(W,S) = (\hat U\hat \Sigma, \hat V^T S)$ is a critical point on the same sub-orbit under $\OG k$. 
This point is can be transported along the orbit $\Theta (W,S)$ 
using $L_{\hat \Sigma^{-1}}$ to $(\hat U, \hat \Sigma\hat V^T S)$.
There exists an orthogonal $Q$ with $\hat U =\bar UQ$ with $\bar U$ formed from $k$ left singular vectors of $X$. 
This gives a canonical point $(\bar U, \Lambda \bar V^T)$ on the same sub-orbit under $\OG k$ as $(\hat U, \hat \Sigma\hat V^T S)$. }
\label{fig:orbit}
\end{figure}
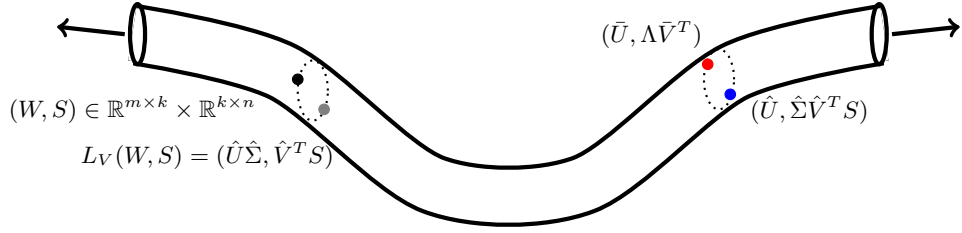

By Theorem \ref{thm:bsp} every critical point with $W$ of positive rank lies on the orbit of a canonical point.
This is depicted schematically in Figure \ref{fig:orbit}.
Note that there is also a family of critical points of the form 
$W= \vzero_{m\by k}$ and $S^T = [C_0^T V_0^T]$ 
with $C_0\in \R^{k\by (n-r)}$.
These are addressed separately below.

The results and proof  of Theorem \ref{thm:bsp} simplify when $q=k$.  
In this case, there is no need for the term $S_b$ and condition (c-ii).  
In contrast, when $q< \min\{k,m\},$ $W$ has a nontrivial null space and
$S$ can be decomposed into the sum of a term $S_a$ in $\cN(W)^\perp$ and a term $S_b$ in $\cN(W)$. The term $S_b$ is redundant since it does not impact the value of $J$, but we need to account for its possible presence.

\subsection{The Critical Points $\mathbf{(\vzero, C_0^T V_0^T)}$}\label{sec:eig_0C0}

We first examine the family of critical points associated with the origin.
These take the form 
\begin{equation}\label{eq:cp@0}
(W,S) = (\vzero, C_0^T V_0^T),
\end{equation}
where $V_0 =[v_{r+1}\ \dots\ v_n] \in \R^{n\by (n-r)}$ and 
$C_0 \in \R^{(n-r)\by k}.$   This analysis is a warm-up for the 
corresponding analysis of the canonical points \eqref{eq:bsp0}. 

Throughout this section $(W,S)$ is a point of the form \eqref{eq:cp@0}.
To verify that $(W,S)$ is a critical point, note that 
$WS=\vzero,$ $W^TE=\vzero$ and $ES^T = X V_0 C_0 =\vzero.$
Hence \eqref{eq:stpoint} is satisfied.
Since $SS^T = C_0^TC_0$ is symmetric positive semidefinite, 
there exists $Z\in \OG k$ and a diagonal matrix 
$\Omega \in \R^{k\by k}$ with diagonal
$\omega_1\geq \cdots \geq \omega_k>0,$ such that 
$C_0^TC_0= Z\Omega Z^T.$ 
We use the columns of $Z$ to define the following matrices. 
For $j\in \Int 1 k,$ let 
$$
G^0_{i,j}= u_i z_j^T,\quad i\in \Int 1 m, \qquad
H^0_{j,i}= z_j v_i^T, \quad i\in \Int 1 n.
$$  
These matrices will identify the eigenvectors and eigenvalues of the Hessian at $(W,S)$.

\begin{theorem} \label{thm:ev0c0} 
(a) For $i\in \Int 1 r$ and $j\in \Int 1 k,$ there exist $\delta_{i,j}, \delta'_{i,j} \in \R$ 
with $\delta_{i,j}\delta'_{i,j}= -1$ such that $(G^0_{i,j}, \delta_{i,j} H^0_{j,i})$ 
and $(G^0_{i,j}, \delta'_{i,j} H^0_{j,i})$ are 
orthogonal eigenvectors of $\nabla^2 J(W, S)$ 
with corresponding eigenvalues
\begin{equation}
\rho_{i,j}  = \textstyle
	\frac{\omega_j}{2} - \sqrt{ \sigma_i^2 +\left(\frac{\omega_j}{2}\right)^2} <0,
		 \quad\textrm{and}\quad
\rho_{i,j}' = \textstyle
	\frac{\omega_j}{2}  + \sqrt{ \sigma_i^2 +\left(\frac{\omega_j}{2}\right)^2} > 0.
		 \label{eq:rho+-}
\end{equation}
(b) $(G^0_{i,j}, \vzero)$ and $(\vzero, H^0_{i,j})$,   
$i \in \Int {r+1} m,$ $j\in \Int{1} k,$ are orthogonal 
eigenvectors of $\nabla^2J(W, S)$ with eigenvalues 
$\omega_j,$ and $0$ respectively.

(c) $(\vzero, H^0_{j,i}),$ $i \in \Int {m+1} n,$ $j\in \Int{1} k,$
is an eigenvector of $\nabla^2J(W, S)$ with eigenvalue $0.$

(d) $\lambda_{\min}(\nabla^2 J(W,S)) = 
	\frac{\omega_k}{2} - \sqrt{ \sigma_1^2 
	+\left(\frac{\omega_k}{2}\right)^2} <0.$
\end{theorem}
\begin{proof}
From \eqref{eq:hessianmap} and the above definitions,
$\nabla^2 J(W,S)[(G,H)] = (G Z \Omega Z^T -XH^T, -G^T X)$.\\
(a) $\nabla^2 J(W,S)[(G^0_{i,j}, \delta H^0_{j,i})] 
= (u_i z_j^T Z \Omega Z^T -\delta Xv_i z_j^T, - z_j u_i^T X )
= ((\omega_j - \delta \sigma_i) G^0_{i,j}, -\sigma_i H^0_{j,i}).$
This has the form $\rho(G^0_{i,j},\delta H^0_{j,i})$ 
if and only if 
\begin{equation}\label{eq:rho1}
\rho = \omega_j -\delta \sigma_i = -\sigma_i/\delta.
\end{equation}
Since $\sigma_i>0,$ this equivalent to
$\delta^2 - \frac{\omega_j}{\sigma_i} \delta - 1 =0,$ yielding real solutions
$\delta_{i,j} = \frac{\omega_j}{\sigma_i} + \frac{1}{\sigma_i}\sqrt{\sigma_i^2 + \left (\frac{\omega_i}{2}\right)^2}$ and 
$\delta'_{i,j} = \frac{\omega_j}{2\sigma_i} -  \frac{1}{\sigma_i}\sqrt{\sigma_i^2 + \left (\frac{\omega_i}{2}\right)^2}.$ Set $a= \frac{\omega_j}{2\sigma_i}$ and
$b = \frac{1}{\sigma_i}\sqrt{\sigma_i^2 + \left (\frac{\omega_i}{2}\right)^2}.$
Then $\delta_{i,j} = a+b$ and $\delta'_{i,j} =a-b$. 
Simple algebra verifies that $\delta_{i,j}\delta'_{i,j} = a^2 - b^2 =-1.$
Orthogonality of the eigenvectors then follows by Lemmas \ref{lem:orthvects} and \ref{lem:aap}.
Substituting $\delta_{i,j}$ and $\delta'_{i,j}$ into \eqref{eq:rho1}
yields \eqref{eq:rho+-}.

(b) For $i\in \Int {r+1} m,$ $\sigma_i=0$. Hence $u_i^T X =\vzero$ and $Xv_i = \vzero.$ 
Then $\nabla^2 J(W,S)[(G^0_{i,j}, \vzero)]= (u_i z_j^T Z \Omega Z^T, - z_j u_i^T X)
= \omega_j ( G^0_{i,j}, \vzero),$ and 
$\nabla^2 J(W,S)[(\vzero, H^0_{j,i})] = (-X v_i z_j^T, \vzero) = 0 (\vzero, H^0_{j,i}).$

(c) For $i\in \Int {m+1} n,$ $j\in \Int 1 k,$
$\nabla^2 J(W,S) [(\vzero, H^0_{j,i})] = (-X v_i z_j^T, \vzero) = 0(\vzero, H^0{j,i}).$

(d) Each negative eigenvalue in \eqref{eq:rho+-} 
has the form $\rho=-\sqrt{x^2 + y^2} +y$ with $y>0$. 
Taking the derivative w.r.t $y$ yields,
$1 -y /\sqrt{x^2+y^2} >0,$ 
i.e., $\rho$ is monotonically increasing in $y.$
Hence the minimum eigenvalue is attained with $i=1$ and $j=k.$ 
\end{proof}

Theorem \ref{sec:eig_cps} part (a) confirms (as expected from \cite{baldi1989neural}) that the critical points \eqref{eq:cp@0}
are all strict saddles. 
The new contribution is the explicit formulas for the eigenvalues
of the Hessian and the insight provided by part (d) into 
how the least eigenvalue changes over this family of critical points. 
Specifically, $\lambda_{\min}(\nabla^2 J(W,S)) = -\sigma_1$
when $\omega_k=0,$ and is monotone increasing as 
$\omega_k$ increases, asymptotically to zero as $\omega_k\ra \infty.$ 
So without a constraint on the size of $\omega_k,$ 
there is no $\gamma<0$ such that 
$\lambda_{\min}(\nabla^2 J(W,S)) < \gamma $ 
over this family of strict saddles.

\subsection{The Canonical Points}  \label{sec:eig_cps}

We now examine the landscape around the canonical points specified 
by \eqref{eq:bsp}. The analysis builds on that given in \S\ref{sec:eig_0C0}. Complete details are given in Appendix \ref{sec:HE@CP}. Here we summarize and discuss the main results.  

We say that a canonical point   
is \nomiss\ if $\lambda_j =\sigma_j$, $j\in \Int 1 q$.
This holds if and only if the $q$ columns of $\bar U$ are 
left singular vectors of $X$ for a set of its $q$ largest singular values.
If this does not hold, then we say that is $(W,S)$ \missatp.
In this case, there exists a least integer $p\in \Int 1 q$ such that 
$\lambda_p <\sigma_p$. 

\begin{theorem} \label{thm:bsp_class}
Let  $1\leq q \leq \min\{k,m\}$ and $(W_c, S_c)$ be a canonical point of the form \eqref{eq:bsp}. If $(W_c,S_c)$ is not maximal, let $p\in \Int 1 q$ denote the least 
integer with $\lambda_p<\sigma_p.$

(a) If $q=m$, or $q=k < m$ and $(W_c,S_c)$ is \nomiss, 
then $(W_c,S_c)$ is a global minimum of $J$. 

(b) If $q=k < m$ and $(W_c, S_c)$ is  \missatp, 
then $(W_c,S_c)$ is a strict saddle with 
	\begin{equation}\label{eq:GCSS_T2}
	\textstyle
	\!\!\!\!
	\lambda_{\min} (\nabla^2 J(W_c,S_c) ) = 
	 - \frac{\sigma_p^2-\lambda_k^2}
	{\frac{\lambda_k^2+1}{2}  +
	\sqrt{\sigma_p^2 + \left(\frac{\lambda_k^2-1}{2}\right)^2} 
	},
	\end{equation}
(c) If $q < \min\{k,m\}$, then $(W_c,S_c)$ is a strict saddle
with $\lambda_{\min} (\nabla^2 J(W_c,S_c) )$ given by
\begin{align}\label{eq:GCSS_T1}
\kern-8pt
\begin{cases}
	\textstyle \frac{\omega_{k-q}}{2} - \sqrt{ \sigma_{q+1}^2 
	+\left(\frac{\omega_{k-q}}{2}\right)^2}, & \!\!\!(W_c,S_c) \textrm{ maximal;}\\
\min \Big \{ \textstyle \frac{\omega_{k-q}}{2} - \sqrt{\sigma_{p}^2 +\left(\frac{\omega_{k-q}}{2}\right)^2},
	\ \shalf \! \left (\lambda_q^2 +1 
	- \sqrt{\left(\lambda_q^2 - 1\right)^2 + 4\sigma_p^2} \right )
 \Big \}, &\textrm{\!\!\!otherwise.}
\end{cases}
\end{align}
\end{theorem}

\begin{proof}
(a) If $q=m,$ then by Theorem  \ref{thm:bsp} 
$J (W_c, S_c) = \sum_{i=1}^r \sigma_i^2 - \sum_{j=1}^m \lambda^2_j  
= - \sum_{i=r+1}^r \sigma_i^2 = 0$. 
If $q=k<m$ and $(W_c,S_c)$ be \nomiss, then $\{\lambda_j\}_{j=1}^k$ is 
a set of $k$ largest singular values of $X.$
Hence 
$J (W_c, S_c) = \sum_{i=1}^m \sigma_i^2 - \sum_{j=1}^k \lambda^2_j 
=  \sum_{i=k+1}^m \sigma_i^2$ achieves its lower bound.

(b) A detailed analysis of this situation is given in \S\ref{sec:eig_frcp}. 
See Corollary \ref{cor:T2_lambda_min} with $a=1$.

(c) A detailed analysis of this situation is given in \S\ref{sec:cp_q<k}. 
See equations \eqref{eq:lam_min_q<k} and \eqref{eq:lam_min_q<k_not}.
\end{proof}

Theorem \ref{thm:bsp_class} confirms that the canonical points \eqref{eq:bsp} are either global minima or strict saddles. 
By Theorem \ref{thm:orbits} the same holds for all critical 
points. This is the main result in \cite{baldi1989neural}.
The new contributions are the explicit formulas for the eigenvalues
of the Hessian (see Appendix \ref{sec:HE@CP} for full details) 
and for the minimum eigenvalue of the Hessian at a canonical strict saddle (parts (b) and (c)).
The latter expressions give insight into how the least eigenvalue 
of the Hessian changes over the family of canonical points. 
For example, if the rank of $W_c$ is $\min\{k,m\}$ 
(no zero columns in $W$), the minimum eigenvalue of the Hessian depends on the separation of the squared values of two distinct nonzero singular values of $X,$ $\sigma_p$ and $\lambda_k.$ In the best case,
$\sigma_p$ is large and $\lambda_k$ is very small and the minimum eigenvalue is approximately $-\sigma_p.$ 
This minimum value increases as $\lambda_k$ increases.
For typical data, this suggests that the worst case is when 
$k$ is large, $\sigma_p=\sigma_k$ and $\lambda_k=\sigma_{k+1}.$
 
In contrast, when the rank of $W_c$ is $q< \min\{k,m\}$ 
(there are zero columns in $W_c$), and $(W_c,S_c)$ is \nomiss,
the minimum eigenvalue depends only the size of a nonzero singular value of $X$ and a singular value of the matrix $C_0^T C_0.$ 
This is given by the first line in \eqref{eq:GCSS_T1}.
Notice the similarity of this equation to the first equation in 
\eqref{eq:rho+-}. The same caveats given there, also apply here.
When $(W_c,S_c)$ is not maximal, there are two ways that negative
eigenvalues can arise and this results in two expressions competing
to provide the least eigenvalue. This is displayed the second line of
\eqref{eq:GCSS_T1}. When $\omega_{k-q}=0,$ the first term reduces to $-\sigma_p$ and this is the minimum of the two terms. As the value
$\omega_{k-q}$ increases so does the first term. Eventually, the second term is the least and this is a constant that 
does not depend on  $\omega_{k-q}.$

Using Theorem \ref{thm:orbits} we can give a negative upper bound on 
$\lambda_{\min}\nabla^2 J$ for any strict saddle $(W,S).$

\begin{corollary} \label{thm:s-pt-ebound}
Let $(W_c,S_c)$ be a canonical strict saddle and $A\in \GLG k.$
Then 
\begin{equation}
\label{eq:GSS_eigbnd}
\lambda_{\min} (\nabla^2 J(L_A(W_c,S_c)) ) \leq 
\frac{\lambda_{\min} (\nabla^2 J(W_c,S_c) )}{\max\{ \lambda_{\max}(AA^T), \lambda_{\min}^{-1} (AA^T)\} } <0. 
\end{equation}
\end{corollary}
\begin{proof}
By Theorem \ref{thm:orbits}, dividing the negative upper bound on 
$\lambda_{\min}\nabla^2 J(W_c,S_c)$ in Theorem \ref{thm:bsp_class} 
by  $\vertt{L_A}^2$ gives a negative upper bound on 
$\lambda_{\min}(\nabla^2J(L_A(W_c,S_c)))$. 
\end{proof}

\section{Negative Upper Bound on $\mathbf{\lambda_{\min} (\nabla^2J)}$ Over Strict Saddles}
  
For a \nomiss\ canonical point $(W_c,S_c)$ with $q<\min\{k,m\}$, 
$\lambda_{\min}(\nabla^2 J(W_c,S_c))$ given in the first line of \eqref{eq:GCSS_T1} depends on a singular value of $C_0.$ 
Since there is no a priori bound on the singular values of $C_0$, 
there is no uniform negative upper bound for 
$\lambda_{\min}(\nabla^2J)$ over all canonical strict saddles.
By this we mean that for no $\gamma<0$ is it the case that 
$\lambda_{\min}(\nabla^2J) < \gamma$ 
for all canonical strict saddles. 
Even if $C_0=\vzero$, there is a second issue.
Since $\vertt{A}^2$ is unbounded over $\GLG k,$
the bound in Corollary \ref{thm:s-pt-ebound} can be arbitrarily close
to $0.$. That does not say, however, that a uniform bound does not exist. 
This section examines this issue.

We first prove the negative result that even when $C_0=\vzero$, 
there is no uniform negative upper bound for $\lambda_{\min}(\nabla^2J(W,S))$ over all strict saddles.  

\begin{theorem} \label{thm:nounub}
For given $k < m,$ and any canonical strict saddle 
$(W_c,S_c)$ with $\rank(W_c)=k,$ 
there is no $\gamma <0$ such that 
$\lambda_{\min}(\nabla^2J(W,S)) \leq \gamma$ 
for all $(W,S)\in \GLGorb(W_c,S_c)$.
\end{theorem}

\begin{proof}
Consider the curve of strict saddles $\{(W_c a, a^{-1} S_c), a>0\}$ 
in $\GLGorb(W_c,S_c)$. 
Corollary \ref{cor:T2_lambda_min} gives an expression for 
$\lambda_{\min} (\nabla^2J(W_ca, a^{-1} S_c)).$
The expression indicates that   
$\lambda_{\min} (\nabla^2J(W_ca, a^{-1} S_c))$ can be made 
arbitrary close to $0$ by making $a$ sufficiently 
small or sufficiently large.
\end{proof}

\subsection{A Uniform Negative Bound for $\mathbf{\lambda_{\min} (\nabla^2J)}$ 
Over Strict Saddles in $\cM_\vzero$} \label{sec:manifold}

We now restrict attention to an interesting 
subset of $\WS$ and show that there is a uniform negative 
upper bound for the minimum eigenvalue of the Hessian at all critical points in this subset.

Let $\cM_C \triangleq \{ (W,S)\colon W^TW-SS^T =C\},$ 
where $C\in \R^{ k\by k}$ is a symmetric matrix. 
$\cM_C$ is of interest for several reasons.
First, the factorization problem \eqref{eq:J} permits imbalance between 
$W$ and $S$ in the sense  that $A\in \GLG k$ can make $WA$ very large (resp. small) while making $A^{-1} S$ very small (resp. large) without changing the value of the objective. However, if $(W,S)\in \cM_C$, the difference between the norms of $W$ and $S$ is bounded: $\|W\|_F^2 - \|S\|_F^2 = \trace(C)$. 
In particular, if $C=\vzero$, $\|W\|_F^2= \|S\|_F^2$.  
This is referred to as a balance condition \cite{AroraICML2018a, DuNIPS2018}.
Second, the term $C_0^T V_0^T$  in \eqref{eq:bsp} is redundant 
and we have no a priori bound on its value.  
We show that for critical points in $\cM_\vzero$, $C_0=\vzero.$ 
Third, it is known that the $\cM_C$ is invariant under gradient flow. 
An initial value for $(W,S)$ specifies $C$, and the gradient flow o.d.e. 
$(W_t, S_t) = -\nabla J(W_t, S_t))$ ensures $(W_t, S_t) \in \cM_C$ 
for $t\geq 0$ \cite[Theorem 1]{AroraICML2018a}.
Lemma \ref{lem:invariance} gives a self-contained proof of this result.

Motivated by the above, we now focus on critical points 
in $\cM_\vzero$.  
Each critical point in $\cM_\vzero$ must be in the orbit of 
some canonical point. We show below that the orbit of a 
canonical point intersects $\cM_\vzero$ 
if and only if $\Lambda$ is invertible and $C_0=\vzero$.

\begin{theorem} \label{thm:M0intersection}
For a canonical point $(W_c, S_c)$ of the form \eqref{eq:bsp},

(a) There exists $A \in \GLG k$ such that $L_A(W_c, S_c) \in \cM_\vzero$ 
if and only if $\Lambda$ is invertible and $C_0=\vzero$. 

(b) If  $L_A(W_c, S_c) \in \cM_\vzero$, then for each $Q\in \OG k$, 
$L_{AQ}(W_c, S_c) \in \cM_\vzero$.
\end{theorem}
\begin{proof}
Let $(W_c, S_c)$ be a canonical point and $A \in \GLG k$. 
$L_A(W_c, S_c) \in \cM_\vzero$ if and only if 
$A^TW_c^T W_c A = A^{-1}S_c S_c^T A^{-T}.$ 
Since $A \in \GLG k$, this is equivalent to 
\begin{equation} \label{eq:m0constraint}
AA^T(W_c^T W_c) AA^T = S_c S_c^T.
\end{equation}
(If) Assume $\Lambda\in \GLG q$ and $C_0=\vzero$. 
For $A =\bigl[\begin{smallmatrix} \sqrt{\Lambda} & \vzero\\ \vzero & I_{k-q}\end{smallmatrix}\bigr]
\in \GLG k,$ 
$AA^TW_c^TW_cAA^T  = 
\bigl[ \begin{smallmatrix} \Lambda & \vzero\\ \vzero & I_{k-q}\end{smallmatrix} \bigr]
\bigl[ \begin{smallmatrix} I_q& \vzero\\ \vzero & \vzero\end{smallmatrix} \bigr]
\bigl[ \begin{smallmatrix} \Lambda & \vzero\\ \vzero & I_{k-q}\end{smallmatrix} \bigr]
= \bigl[ \begin{smallmatrix} \Lambda^2 & \vzero\\ \vzero & \vzero\end{smallmatrix} \bigr]
= S_cS_c^T.
$ 
Hence $L_A(W_c, S_c)  \in \cM_\vzero.$

(Only If) 
There exists $A \in \GLG k,$ with $L_A(W_c,S_c)\in \cM_\vzero.$ 
So $A$ satisfies \eqref{eq:m0constraint}. 
In general,
$W_c^T  W_c = 
\bigl[\begin{smallmatrix} I_q & \vzero\\ \vzero & \vzero \end{smallmatrix}\bigr]$
and
$S_c S_c^T 
= \bigl[\begin{smallmatrix} \Lambda^2 & \vzero \\ \vzero & C_0^TC_0 
\end{smallmatrix}\bigr].
$
Let $R=AA^T \succ \vzero$ and write 
$
R = \bigl[\begin{smallmatrix}
R_1 & R_3 \\
R_3^T & R_2
\end{smallmatrix}\bigr]. 
$
By \eqref{eq:m0constraint}, $R$ satisfies 
$R \bigl[\begin{smallmatrix} I_q & \vzero\\ \vzero & \vzero \end{smallmatrix}\bigr] 
R
=
\bigl[\begin{smallmatrix}
R_1^2 & R_1 R_3 \\
R_3^T R_1 & R_3^T R_3
\end{smallmatrix}\bigr] 
\!=\! 
\bigl[\begin{smallmatrix} \Lambda^2 & \vzero \\ 
\vzero & C_0^T C_0 \end{smallmatrix}\bigr].$
It follows that $R_1^2 = \Lambda^2$, 
$R_1R_3 = \vzero $, and 
$R_3^T R_3 = C_0^TC_0$.
Since $ \Lambda^2$ is diagonal and non-negative, 
the first requirement gives $R_1 = \Lambda$.
Since $R\succ \vzero$, we must have $R_1\succ \vzero$, and hence 
$\Lambda\in \GLG q$. 
The second requirement then gives $R_3 = \vzero$, and
the third implies $C_0^T  C_0=\vzero.$ This gives $\|C_0\|_F^2=0$ and hence $C_0=\vzero$.

(b) If $A\in \GLG k$ satisfies \eqref{eq:m0constraint} and $Q \in \OG k,$
then $(AQ)(AQ)^T = AA^T$. So $AQ$ satisfies \eqref{eq:m0constraint}.
\end{proof}

For the rest of this section we only consider canonical points with  $C_0=\vzero$ and $\Lambda\in \GLG q$. 
The second condition requires $q\leq \min\{k,r\}$.
This allows for factorization with $k \leq r,$ and with $k > r.$
Under these assumptions we can apply $L_{\Lambda^{\nhalf}}$ 
to map the canonical point in \eqref{eq:bsp} into $\cM_\vzero$. 
This yields
\begin{equation}\label{eq:bc_ccp}
(W_0,S_0) = \left (
\begin{bmatrix} \bar U\Lambda^{\nhalf}  & \vzero_{m\by(k-q)}\end{bmatrix} , 
\begin{bmatrix} \Lambda^{\nhalf} \bar V^T \\ \vzero_{(k-q)\by n} \end{bmatrix} \right ).
\end{equation} 

\begin{corollary}\label{cor:canpM0}
The set of critical points in $\cM_\vzero$ is 
$\{\OGorb(W_0,S_0)\colon  (W_0,S_0) \textrm{ has the form \eqref{eq:bc_ccp} }\}$.
\end{corollary} 
\begin{proof}
This follows from \eqref{eq:bc_ccp} and part (b) of Theorem \eqref{thm:M0intersection}.
\end{proof}

By Corollary \ref{cor:canpM0} and Theorem \ref{thm:bsp}, to determine the landscape around a critical point on $\cM_\vzero$ we need only examine the landscape around 
the point $(W_0,S_0)$ in \eqref{eq:bc_ccp}. Each such $(W_0,S_0)$ lies in the orbit of a companion canonical point $(W_c,S_c)$ sharing the same $\bar U$ and $\Lambda.$

\begin{theorem} \label{thm:acp}
Consider the critical point $(W_0,S_0)\in \cM_\vzero$ given by \eqref{eq:bc_ccp}.
If $(W_0,S_0)$ is not maximal, let $p\in \Int 1 q$ denote the least 
integer with $\lambda_p<\sigma_p.$

(a) If $q=\min\{k,r\}$ and $(W_0,S_0)$ is \nomiss, then
it is a global minimum. 

(b) If $q=\min\{k,r\}$ and $(W_0,S_0)$ is \missatp, then 
$\lambda_{\min}(\nabla^2 J(W_0, S_0) )  =  - ( \sigma_p - \lambda_k ).$

(c) If $q<\min\{k,r\}$, then 
$
\lambda_{\min}(\nabla^2 J(W_0, S_0) ) 
= 
\begin{cases}
 -\sigma_{q+1}& \textrm{if } (W_0,S_0) \textrm{ is \nomiss};\\ 
 - \sigma_p , &\textrm{if } (W_0,S_0) \textrm{ \missatp.} 
\end{cases}
$
\end{theorem}
\begin{proof}
Let $(W_0,S_0)$ be in the orbit of the canonical point $(W_c,S_c).$

(a) Since $(W_0,S_0)$ is maximal so is $(W_c,S_c).$
Then by Theorem \eqref{thm:bsp_class}, $(W_c,S_c)$ is a global minimum.
Hence $(W_0,S_0)$ is a global minimum. 

(b) $q=\min\{k,r\}$ and $(W_0, S_0)$ \missatp\ implies $q=k < r$. 
By Lemma \ref{lem:g2}, for each $i\in \Int 1r$ and $j\in \Int 1k,$ with 
$u_i^T\bar U=\vzero$ and $\lambda_j<\sigma_i,$ 
$\nabla^2J(W_c,S_c)$ has a negative eigenvalue. 
For each such pair $i,j$ we show that $\nabla^2 J(W_0,S_0)$ 
also has a negative eigenvalue.
Let  $G_{i,j} =  u_i \ve_j^T$ and $H_{j,i}= \ve_j v_i^T.$
Since $\bar U^T u_i =\vzero$ and $\bar V^T v_i =\vzero,$ we have
$H_{j,i} S_0 = \vzero$ and $W_0^TG_{i,j}=\vzero.$ 
Using \eqref{eq:D2J} and $W_0^TW_0 = S_0S_0^T =\Lambda$ we have 
$\nabla^2 J(W_0,S_0)[(G_{i,j},H_{j,i})] 
 = (u_i\ve_j^T\Lambda 
+ (W_0S_0 -X) v_i\ve_j^T,
\Lambda \ve_j v_i^T  
+ \ve_j u_i (W_0S_0-X))
= (\lambda_j -\sigma_i) (G_{i,j}, H_{j,i}).
$
Thus $-(\sigma_i - \lambda_j)$ is a negative eigenvalue of $\nabla^2 J(W_0,S_0)$.
By Theorem \ref{thm:orbits} part (b), these are the only negative eigenvalues 
of $\nabla^2 J(W_0,S_0).$ 
Hence $\lambda_{\min}(\nabla^2 J(W_0, S_0) )  =  - ( \sigma_p - \lambda_k)$.

(c) Let $q<\min\{k,r\}$ and $(W_c,S_c)$ be \nomiss. 
Lemma \ref{lem:qkC0} shows that 
for $i\in \Int 1r,$ $j\in \Int {q+1} k$ with $u_i^T \bar U=\vzero,$ 
$(G_{i,j}, H_{j,i})$ is an eigenvector of $\nabla^2 J(W_c,S_c)$ 
with eigenvalue $-\sigma_i.$  Moreover, these are the only negative 
eigenvalues.
The same proof for the same pairs $i,j,$ shows that $(G_{i,j}, H_{j,i})$ 
is also an eigenvector of $\nabla^2J(W_0,S_0)$ with eigenvalue 
$-\sigma_i.$  
By Theorem \ref{thm:orbits} part (b) these are the only negative 
eigenvalues of $\nabla^2 J(W_0,S_0)$. Hence in this case, 
$\lambda_{\min}(\nabla^2 J(W_0,S_0) ) = - \sigma_{q+1}$.

Now suppose $(W_c,S_c)$ is \missatp.
Then $\nabla^2J(W_c,S_c)$ has two groups of negative eigenvalues
and these are the only negative eigenvalues. 
Group 1:
By Lemma \ref{lem:g2} and Lemma \ref{lemma:relate-evals}, for each 
$i\in \Int 1 r$ and $j\in \Int 1q,$ with $u_i^T\bar U = \vzero$ and $\sigma_i <\lambda_j$, 
$\nabla^2 J(W_c,S_c)$ has a negative eigenvalue.
Group 2:
By Lemma \ref{lem:qkC0}, for each 
$i\in \Int 1 r$ and $j\in \Int {q+1} k,$ with $u_i^T\bar U = \vzero,$ 
$\nabla^2 J(W_c,S_c)$ has a negative eigenvalue. 
We show that for each pair of indices in each group,
$\nabla^2J(W_0,S_0)$ has a negative eigenvalue.
First, by adapting the result in part (b) to the current situation, we see that $-(\sigma_i - \lambda_j)$ is a negative eigenvalue of 
$\nabla^2 J(W_0,S_0)$ for all pairs of indices $i,j$ in group 1. 
Second, the first result in part (c) shows that $-\sigma_i$ is a negative 
eigenvalue of $\nabla^2 J(W_0,S_0)$ for all pairs indices $i,j$ in group 2.   
By Theorem \ref{thm:orbits} part (b), these are all of negative eigenvalues of $\nabla^2 J(W_0,S_0).$ The least eigenvalue is attained in the second group by selecting $i=p$ and $j=q+1$.
Hence in this case, 
$\lambda_{\min}(\nabla^2 J(W_0,S_0) ) = - \sigma_{p}$.
\end{proof}

Theorem \ref{thm:acp} provides an expression for 
$\lambda_{\min}(\nabla^2J (W_0,S_0))$ at each strict saddle 
$(W_0,S_0)\in \cM_\vzero$. If $q=k<r$ and $(W_0,S_0)$ is 
\nomiss, this value is a difference of two singular values. 
For example, if the first $k+1$ singular values are distinct, 
then the largest possible value is $\Delta_k =  - \min_{1\leq j\leq k} \sigma_j - \sigma_{j+1}.$
If the singular values are not distinct, the corresponding value is 
based on differences of the consecutive {\em distinct} values assumed by the $\sigma_j$.
When $q< r <k,$ part (b) of the Theorem indicates that largest bound is 
$-\sigma_r$ independent of $k.$

\begin{corollary}\label{cor:ubnd}
There exists $\gamma<0$ such that 
$\lambda_{\min}(\nabla^2J(W_0,S_0)) <\gamma$  for all strict saddles in 
$(W_0,S_0) \in \cM_\vzero.$
\end{corollary}
\begin{proof}
The minimum eigenvalues in Theorem \ref{thm:acp} 
can only take a finite set of negative values.
\end{proof}

For $q<k \leq r$ and a generic distribution of singular values, 
we expect $\Delta_k$ to be the larger bound when $k$ is small
and $-\sigma_k$ to be the larger bound for larger $k$.

\section{Conclusion}\label{sec:conclusion}
Our main contribution is to provide a more complete 
understanding of the landscape of simple matrix factorization. 
Our approach considers the orbits of critical points under the general linear group, and represents each orbit by a canonical point.
Prior work tells us that a critical points of \eqref{eq:J} are either
global minima or a strict saddles \cite{baldi1989neural}. 
We go beyond that result to determine the eigenvalues and eigenvectors of the Hessian at each canonical point. 
This determines the number of negative eigenvalues and 
leads to an expression for minimum eigenvalue of the Hessian 
at each canonical point. The latter expression allows us to show that
the minimal eigenvalue was not uniformly bounded below zero 
over all strict saddles. There are two reasons for this. 
First, the matrix $C_0$ that appears, for example, in the family of critical points $(\vzero,C_0^T V_0^T)$ can push the minimum 
eigenvalue of the Hessian towards $0.$
Second, moving a strict saddle along its orbit under $\GLG k$
can also push the minimum eigenvalue of the Hessian towards $0$. However, we show that constraining attention to a particular manifold 
$\cM_\vzero,$ ensures the least eigenvalue of the Hessian at strict saddles is uniformly bounded below zero. We prove this by characterizing the critical points on $\cM_\vzero$ and using this to obtain an explicit expression for $\lambda_{\min}(\nabla^2 J(W,S)$ 
for each strict saddle $(W,S)$ in $\cM_\vzero$. 

The manifold $\cM_\vzero$ is special in that points on the manifold satisfy the so-called balance condition $\|W\|_F^2= \|S\|_F^2.$
Moreover, it is known that $\cM_\vzero$ is invariant under gradient flow \cite{AroraICML2018a, DuNIPS2018}. However, despite its special characteristics, continuity ensures that our results above degrade gracefully as one deviates from this particular manifold.

Finally, our development has used the natural setting of the problem, made no assumptions of symmetry or artificially created symmetry, and has avoided vectorization of the relevant differentials. We believe that this yields greater clarity and insight. For example, it permits us determine eigenvectors with a simple interpretable structure at every canonical critical point. We also posit that this approach is more amenable to generalization to related problems.

\appendix

\section{Auxiliary Lemmas and Proofs}\label{sec:auxlem}

\begin{proof}[{\bf Proof of Lemma \ref{lem:LAderiv}}]
(a) $\nabla J(L_A(W,S)) = (ES^TA^{-T}, A^TW^TE) = L_{A^{-T}} \nabla J(W,S).$ \\
(b) $DJ(L_A(W,S))[(G,H)] = \ip{L_{A^{-T}}(\nabla J(W,S))}{(G,H)} 
= \ip{(\nabla J(W,S))}{L_{A^{-1}}(G,H)}.$\\
(c) $\nabla^2 J(L_A(W,S)) [(G,H)] 
= ((GA^{-1}) SS^T A^{-T} + W(AH)S^T A^{-T} + EH^T, A^TW^TW(AH) +A^TW^T(GA^{-1}) S +G^TE)
= L_{A^{-T}} \left ( \nabla^2 J(W,S)[L_{A^{-1}}(G,H)]\right ).$\\
(d) By definition, $D^2J(L_A(W,S))[(G,H)] = \ip{\nabla^2(L_A(W,S))[(G,H)]}{(G,H)}.$
Using part (c), the RHS of the previous equation can be written as 
$\ip{L_{A^{-T}}(\nabla^2(W,S)[L_{A^{-1}}(G,H)] )}{(G,H)}
= \ip{\nabla^2(W,S)[L_{A^{-1}}(G,H)]}{L_{A^{-1}}(G,H)}
=D^2 J(W,S) [L_{A^{-1}} (G,H)].$
\end{proof}

\begin{proof}[{\bf Proof of Theorem \ref{thm:orbits}}]
(a) (i) If $\nabla J(W,S)=\vzero$, then by part (a) of Lemma \ref{lem:LAderiv} and the linearity of $L_{A^{-T}}$, $\nabla J(L_A(W,S))=\vzero$.  
(ii) Let $\nabla^2J(W,S)[(G,H)] = \lambda (G,H)$ with $\lambda<0.$
Then $D^2(W,S)[(G,H)]=\ip{\nabla^2(W,S)[(G,H)]}{(G,H)} < 0$.
By part (d) of Lemma \ref{lem:LAderiv}, 
$D^2J(L_A(W,S)) [L_A(G,H)] = D^2(W,S)[(G,H)] <0.$
Thus $\nabla^2 J(L_A(W,S))$ also has a negative eigenvalue.\\
(b) Fix an orthonormal basis $\{(G_i, H_i)\}_{i=1}^{k(m+n)}$ for $\WS.$
With respect to this basis, each $(G,H)\in \WS$ has a unique coordinate vector 
$g=\phi(G,H)\in \R^{k(m+n)},$
and each linear map $L_A\colon \WS\to \WS,$ with $A\in \GLG k,$ 
has a unique matrix representation $M(L_A)\in \R^{k(m+n)}.$ 
We then have $g_A = M(L_A) g$ where $g_A = \phi(L_A(G,H)).$  
Inner products are preserved using coordinates: 
$\textstyle \ip{(G,H)}{(G',H')} 
= \ip{g}{g'}.$
It is also easy to verify that $\ip{L_{A^T}(G,H)}{(G',H')} 
= \ip{(G,H)}{L_A(G',H')}.$ We then have
\begin{align*}
\ip{L_{A^T}(G,H)}{(G',H')} &= \ip{M(L_{A^T})g}{g'} = g^T M(L_{A^T})^T g', 
\quad \textrm{and} \\
\ip{(G,H)}{L_A(G',H')} &= \ip{g}{M(L_A)g'} = g^T M(L_A) g'.
\end{align*}
These expressions must be equal for all $g, g'$. Thus $M(L_{A^T})^T = M(L_A).$
Since $\nabla^2 J(W,S)$ is a linear map on $\WS$, it has a matrix $P$ in the given basis. Similarly, $\nabla^2 J(L_A(W,S))$ has a matrix $Q$.
Since $P$ and $Q$ represent Hessian maps,  both are symmetric matrices. 
By part (c) of Lemma \ref{lem:LAderiv},
$\nabla^2 J(L_A(W,S))[(G,H)] = L_{A^{-T}}( \nabla^2 J(W,S)[L_{A^{-1}}(G,H)] ).$
Letting $g=\phi(G,H),$ and writing this equation using coordinates yields
$Q g = M (L_{A^{-1}})^T P M(A^{-1}) g.$ 
Since this must hold for all $g$, we conclude that 
$Q = M (L_{A^{-1}})^T P M(L_{A^{-1}}).$ 
Thus the matrices $P$ and $Q$ are congruent. 
The result then follows by Sylvester's theorem 
of inertia \cite[Theorem 4.5.8]{Horn2013}.

(c) Let $(G,H)$ be an eigenvector of $\nabla^2 J(W,S)$ 
with eigenvalue $\lambda$. By Lemma \ref{lem:LAderiv} part (c), 
\begin{equation*}
\nabla^2J(L_A(W, S)) [L_A(G,H)] = \lambda L_{A^{-T}}( (G,H) ) 
	= \lambda (G A^{-T}, A^T H).
\end{equation*}
If $A\in \OG k$ then $A^{-1}=A^T$. 
Hence $\nabla^2J(W,S) [L_A(G,H)] = \lambda L_A(G ,H)$. 
Thus $\lambda$ is also an eigenvalue of $\nabla^2J(WA,A^{-1}S)$. 
A symmetric argument proves the converse result.

(d) Let $\lambda = \lambda_{\min}(\nabla^2 J(W,S))$ have eigenvector $(G,H)$. 
By Lemma \ref{lem:LAderiv} part (c), 
\begin{align*}
D^2 J(L_A(W,S))[L_A(G,H)] 
& = \lambda \ip{(G A^{-T}, A^T H)}{(GA,A^{-1}H)}
 =\lambda \trace( G^T G + HH^T ).
\end{align*}
Dividing the above equation by the squared norm of $(GA,A^{-1}H)$ yields 
\begin{align*}
\lambda_{\min} (\nabla^2 J(L_A (W, S) ) )
\leq  \lambda
\frac{\trace(G^T G + HH^T)}{\trace(A^TG^TGA + A^{-1} HH^TA^{-T})}
\leq 
\lambda /\vertt{L_A}^2 ,
\end{align*}
where $\vertt{L_A}$ denotes the induced norm of $L_A.$
The result then follows by Lemma \ref{lem:INLA}.
\end{proof}

\begin{lemma}\label{lem:INLA}
$\vertt{L_A} = \max\{\lambda_{\max}^{\nhalf} (AA^T) , \lambda_{\min}^{-\nhalf} (AA^T) \}$. 
\end{lemma}
\begin{proof}
$\vertt{L_A} = \max_{\|(G,H)\|_F =1} \|(GA, A^{-1} H)\|_F$.
Let $u$ (resp. $v$) be a unit norm eigenvector of $AA^T$ (resp. $(AA^T)^{-1}$) 
with eigenvalue
$\lambda_{\max} \triangleq \lambda_{\max}(AA^T)$ 
(resp. $\mu_{\max} \triangleq \lambda_{\min}^{-1}(AA^T)$). 
Let $(G,H)$ satisfy $\|(G,H)\|_F =1$ and maximize 
\begin{equation}\label{eq:objLA}
\|(G,H)\|_F^2 = \trace(GAA^TG^T) + \trace(H^T (AA^T)^{-1} H).
\end{equation}
If a nonzero row of $G$ is replaced by a scaled version of $u^T$ with the same norm,
the constraint remains satisfied and the objective can increase. 
The same holds if a nonzero column of $H$ is replaced by a suitably scaled 
version of $v$. 
Hence there is an optimal $(G,H)$ of the form 
$G=\sum \alpha_i \ve_i u^T$, $H=\sum \beta_j v \ve_j^T$,
$\sum \alpha_i^2 +\sum \beta_j^2 =1$, with optimal value
$(\sum \alpha_i^2) \lambda_{\max} + (\sum \beta_j^2) \mu_{\max}$. 
This value is achieved by $(G^\star, H^\star)= (\alpha \ve_1u^T,
\beta v\ve_1^T)$ with $\alpha^2+\beta^2 =1$.  
Thus the optimal value of \eqref{eq:objLA} is
$
\min_{\alpha^2+ \beta^2=1} ~\alpha^2  \lambda_{\max}  + \beta^2 \mu_{\max}
= \max\{\lambda_{\max}, \mu_{\max} \}$. So $\vertt{L_A} =   \max\{\sqrt{\lambda_{\max}}, \sqrt{\mu_{\max} } \}$.
\end{proof}

The following lemma concerns $S_b\in \R^{(k-q)\by n}$ 
introduced in the proof of Theorem \ref{thm:bsp}. 

\begin{lemma}\label{lem:S2props}
Let the columns of $\bar U$ be a set of $q$ left singular vectors of $X,$
the columns of $\bar V$ be a matching set of right singular 
vectors of $X,$  $\Lambda$ be the diagonal matrix with the 
singular values corresponding to $\bar U$ on the diagonal, 
and $V_0 \triangleq  [v_{r+1}\ \dots\ v_n].$
Then 

(a) $\Sigma V^T V_0=\vzero$ and $\Lambda \bar V^T V_0 = \vzero.$

(b) $(\bar U \bar U^T -I )X {S}_b^T =\vzero \ \Leftrightarrow\ 
S_b^T= \bar V \bar C +V_0C_0$ 
for some $\bar C \in \R^{q\by (k-q)}$, $C_0 \in \R^{(n-r)\by(k-q) }$.
\end{lemma}  

\begin{proof}
(a) For $i\in \Int {m+1} n,$ $\Sigma V^T v_i =0$ since $v_i$ is not a column in $V.$
For $i\in \Int {r+1} m,$ $\sigma_i=0$. 
If $v_i$ is not a column of $V,$ $\Sigma \bar V^T v_i =0.$
If $v_i$ is a column in $V,$  then $\sigma_i=0$ 
and $\Sigma V^T v_i = \Sigma \ve_i = \sigma_i \ve_i = 0.$
The proof for $\bar V$ is almost identical.

(b) (If) $(\bar U \bar U^T -I)X = \bar U \Lambda \bar V^T - U\Sigma V^T$.
Hence $(\bar U \bar U^T -I )X {S}_b^T = (\bar U \Lambda \bar V^T - U\Sigma V^T)
\bar V \bar C + (\bar U \Lambda \bar V^T - U\Sigma V^T)V_0C_0.$
Since the columns in $\bar V$ are a subset of the columns in $V,$ 
the first term is zero. The second term is zero by part (a).\\
(Only If)
$(\bar U \bar U^T -I )X {S}_b^T =\vzero$ implies  
for some $\bar B \in \R^{q\times (k-q)}$, $XS_b = \bar U \bar B$. 
So $U\Sigma V^T S_b^T = \bar U B$. 
Let $1\leq i\leq r$ with $u_i$ not a column of $\bar U$. 
Since $1\leq i\leq r$, $\sigma_i > 0$.
Multiplying both sides of the previous equation by 
$u_i^T$ yields $\sigma_i v_i^T S_b^T = u_i^T \bar U B =\vzero$.  
Hence for every $v_i$ with a nonzero singular value
that is not a column in $\bar V$, $v_i^T S_b^T = 0$. 
Thus the columns of $S_b^T$ lie in the span of the $v_i$ that are either columns of $\bar V$ or of $V_0$. 
Hence there exist $\bar C, C_0$ such that $S_b^T = \bar V \bar C + V_0 C_0$. 
\end{proof}

\section{Hessian Eigenvalues and Eigenvectors}
\label{sec:HE@CP}

Let $\{z_j\}_{j=1}^k$ denote an orthonormal basis in $\R^k,$
$G_{i,j} = u_i z_j^T$, $i \in \Int 1 m$, 
and $H_{j,i} = z_j v_i^T$, $i \in \Int 1 n$. 
Then let
$\cV= \{ (G_{i,j},  \vzero)\colon  j\in \Int 1k, i\in \Int 1 m\} \cup   
\{ (\vzero ,H_{j,i})\colon j\in \Int 1 k, i\in \Int 1 n\}.$

\begin{lemma}\label{lem:orthvects}
$\cV$ is a set of $k(m+n)$ orthonormal vectors in $\WS$. 
\end{lemma}
\begin{proof}
It is clear that $\cV$ contains $k(m+n)$ vectors and that every vector 
in the first subset is orthogonal to every vector in the second subset. 
Each of the two subsets forming $\cV$ is orthonormal. For example,
$\ip{(G_{i,j}, \vzero)}{(G_{i',j'},\vzero)} = \trace(z_{j'}^T z_j u_i^T u_{i'})
= 1$ if $i=i', j=j',$ and is $0$ otherwise.
A similar equation proves the same holds for the second subset.
\end{proof}

\begin{lemma}\label{lem:aap}
Let $\alpha, \alpha' \in \R$ with $\alpha \alpha'= -1$. 
Then any pair of vectors $(G_{i,j}, \vzero)$ and $(\vzero, H_{s,t})$ in $\cV$
can be replaced by the pair of vectors $(G_{i,j}, \alpha H_{s,t})$ and 
$(G_{i,j}, \alpha' H_{s,t})$ without changing the orthogonality of the elements 
in the modified set $\cV^\prime$.
\end{lemma}
\begin{proof}
The new vectors are orthogonal:
$\ip{(G_{i,j}, \alpha H_{s,t})}{(G_{i,j}, \alpha' H_{s,t})}
= \trace(z_j u_i^T u_i z_j) +\alpha \alpha' \trace(v_t z_s^T z_s v_t^T) =0.$ 
So 
$\vspan\left ( (G_{i,j}, \alpha H_{s,t}), (G_{i,j}, \alpha' H_{s,t}) \right ) = 
\vspan \left ( (G_{i,j}, \vzero), (\vzero, H_{s,t})\right ).$ 
Hence $(G_{i,j}, \alpha H_{s,t}), (G_{i,j}, \alpha' H_{s,t})$ are orthogonal
to all other elements of $\cV$.
\end{proof}

\subsection{A Full Rank Canonical Point}\label{sec:eig_frcp}
We now determine the eigenvalues and eigenvectors of the Hessian at a full rank canonical point $(W_c,S_c).$ Here $W_c\in \R^{m\by k}$ has rank $k$. Hence $k<m$.
It will be convenient to derive a slightly more general result by considering the 
curve $\{ (W_c a, a^{-1} S_c)\colon a\in \R, a\neq 0\} \subset \GLGorb(W_c,S_c).$
Each point $(W,S)$ on this curve is a critical point. 
The following three lemmas obtain expressions for the $k(m+n)$ 
eigenvalues and corresponding orthogonal eigenvectors of $\nabla^2 J(W, S)$.
To simplify the exposition we will assume $r\leq m\leq n$. 
Symmetric arguments cover the case $n<m$.
For $j\in \Int 1 k,$ we set $G_{i,j} = u_i \ve_j^T,$ $i\in \Int 1 m,$ 
and $H_{j,i} = \ve_j v_i^T,$ $i\in \Int 1 n.$

\begin{lemma}\label{lem:g1}
For $i\in \Int {m+1} n$ and $j\in \Int 1k,$ $(\vzero, H_{j,i})$ is an eigenvector 
of $\nabla^2J(aW_c,a^{-1}S_c)$ with eigenvalue $\rho_0=a^2$.
\end{lemma}
\begin{proof}
$\nabla^2J(aW_c,a^{-1}S_c)[(\vzero, H_{j,i})]
= (W_cH_{j,i}S_c^T +(W_cS_c-X)(H_{j,i})^T, \ 
	a^2 W_c^TW_c H_{j,i} )
=  (W_c \ve_j v_i^T \bar V \Lambda + (W_cS_c-X)v_i \ve_j^T ,
    a^2 \ve_j v_i^T )
 = a^2 (\vzero, H_{j,i}).$
\end{proof}

Now consider vectors constructed from $G_{i,j}$ and $H_{j,i}$
for $i\in \Int 1m$ with $u_i$ not a column of $\bar U.$
The following lemma separates this into two parts; 
first $i\in \Int 1r,$ then $i\in \Int {r+1}m.$

\begin{lemma}\label{lem:g2}
For $j\in \Int 1k,$ and $i\in \Int 1 m$ with $u_i$ not a column in $\bar U,$ 
the following hold:
(a)
If $i\in \Int 1 r,$ there exist $\alpha_{i,j}, \alpha'_{i,j} \in \R$ with 
$\alpha_{i,j}\alpha'_{i,j}= -1$  such that $(G_{i,j}, \alpha_{i,j} H_{j,i})$ 
and $(G_{i,j}, \alpha'_{i,j} H_{j,i})$ 
are eigenvectors of $\nabla^2 J(W_c a, a^{-1}S_c)$ with corresponding eigenvalues
\begin{equation}
\rho_{i,j} \textstyle =
	\frac{1}{2} \left (\frac{\lambda_j^2 +a^4}{a^2} 
		- \sqrt{\left(\frac{\lambda_j^2 - a^4}{a^2}\right)^2 + 4\sigma_i^2}
		 \right )\quad 
		 \begin{cases}
		>0,  &\textrm{if } \lambda_j > \sigma_i;\\
		= 0, &\textrm{if } \lambda_j = \sigma_i;\\
		<0, &\textrm{if } \lambda_j < \sigma_i.
		\end{cases}
		 \label{eq:rho+}
\end{equation}
\begin{equation}
	\rho'_{i,j} \textstyle =
	 \frac{1}{2} \left (\frac{\lambda_j^2 +a^4}{a^2} 
		+ \sqrt{\left(\frac{\lambda_j^2 - a^4}{a^2}\right)^2 + 4\sigma_i^2} 
		\right )
	 \quad >0. 
\end{equation}
(b) Alternatively, if $i\in \Int {r+1} m,$ 
the pair of vectors $(G_{i,j},  \vzero)$ and $(\vzero, H_{j,i})$, 
are eigenvectors of $\nabla^2J(aW_c,a^{-1}S_c)$ with eigenvalues 
$\rho_j = \lambda_j^2/a^2 \geq 0,$ and $\rho_0=a^2>0,$ respectively.
\end{lemma}
\begin{proof}
(a) Under the stated assumptions,
$\sigma_i > 0,$ $u_i^T\bar U = \vzero$ and $v_i^T\bar V =\vzero.$ 
Let $(G,H) = (u_i\ve_j^T, \alpha \ve_j v_i^T)$ with $\alpha \in \R.$
For $(W,S)=(W_c a, a^{-1}S_c)$ we seek $\alpha, \rho\in \R$ such that $(G,H)$
is an eigenvector of $\nabla^2 J$ with eigenvalue $\rho$. 
Using \eqref{eq:hessianmap}, this is equivalent to
\begin{align}\label{eq:Heigvec_0}
\!\!\!\! GSS^T + WHS^T + (WS -X) H^T = \rho G, \quad
W^TWH + W^TGS +G^T (WS - X) =\rho H. 
\end{align}
In the present context these equations become
$u_i \ve_j^T (a^{-2} \Lambda^2 ) -\alpha X v_i \ve_j^T
=  (a^{-2}\lambda_j^2 - \alpha \sigma_i ) u_i \ve_j^T = \rho u_i \ve_j^T,$ and
$\alpha a^2 \ve_j v_i^T - \ve_j u_i^TX
=  \alpha (a^2-\sigma_i/\alpha )\ve_j v_i^T= \rho \ve_jv_i^T.$
Solving for $\rho$ we obtain 
\begin{equation}\label{eq:rho_alpha_eqs}
\rho = a^{-2}\lambda_j^2 - \alpha \sigma_i \quad\textrm{and}\quad 
\rho = a^2- \sigma_i/\alpha.
\end{equation}
Thus $\nabla^2 J(W_ca,a^{-1}S_c)[(G,H)]=\rho~(G,H)$ 
if and only if $\alpha$ is a real root of the equation
\begin{equation}\label{eq:quad_alpha}
\sigma_i \alpha^2 - \frac{\lambda_j^2- a^4}{a^2}~\alpha - \sigma_i = 0.
\end{equation}   
This equation has two real roots $\alpha^+$ and $\alpha^{-}$ with
$\alpha^\pm \textstyle=\frac{1}{2\sigma_i} \left ( \frac{\lambda_j^2 -a^4}{a^2} \pm 
\sqrt{\left ( \frac{\lambda_j^2 - a^4}{a^2}\right )^2 + 4 \sigma_i^2}~ \right )
	\label{eq:alpha_pm}.
	$
Using the first equation for $\rho$ above and this result
we find
\begin{align}
\!\!\!\!\!
\rho(\alpha^+) &=\textstyle
	\frac{1}{2} \left (\frac{\lambda_j^2 +a^4}{a^2} 
		- \sqrt{\left(\frac{\lambda_j^2 - a^4}{a^2}\right)^2 + 4\sigma_i^2}
		 \right ), \ \textrm{eigenvector: }(u_i\ve_j, \alpha^{+} \ve_j v_i^T), \label{eq:alpha+}\\
\!\!\!\!\!
\rho(\alpha^-) &= \textstyle
	\frac{1}{2} \left (\frac{\lambda_j^2 +a^4}{a^2} 
		+ \sqrt{\left(\frac{\lambda_j^2 - a^4}{a^2}\right)^2 + 4\sigma_i^2} 
		\right ),\  \textrm{eigenvector: } (u_i\ve_j, \alpha^- \ve_j v_i^T).\label{eq:alpha-}
\end{align}
It is readily checked from \eqref{eq:alpha_pm} that $\alpha^+\alpha^- = -1$. 
Hence these eigenvectors have the form $(G_{i,j}, \alpha_{i,j} H_{i,j})$ 
and $(G_{i,j}, \alpha'_{i,j} H_{i,j})$ 
with
$\alpha_{i,j}=\alpha^+$, $\alpha_{i,j}^\prime = \alpha^{-}$ and 
$\alpha_{i,j}\alpha_{i,j}^\prime = -1$.
Let 
\begin{equation}\label{eq:def_cd}
\textstyle
c=\frac{\lambda_j^2 + a^4}{a^2},
\quad  \textrm{ and } \quad
d= \sqrt{\left(\frac{\lambda_j^2 - a^4}{a^2}\right)^2 + 4\sigma_i^2}. 
\end{equation}
Then $c,d > 0,$ $\rho(\alpha^+) = \nhalf (c - d),$ and $\rho(\alpha^-)= \nhalf(c+d)$. 
Clearly $\rho(\alpha^{-}) >0.$
Simple algebra verifies that $c^2-d^2 = -4(\sigma_i^2 - \lambda_j^2).$  
This gives the sign classifications of $\rho_{i,j}$ in \eqref{eq:rho+}.

(b) 
In this case, $\sigma_i=0.$ Hence $u_i^T X =\vzero$ 
and $X v_i =\vzero.$ 
In addition, $u_i^T \bar U =\vzero$ and $v_i^T \bar V = \vzero$.
Hence $u_i^T WS = u_i^T\bar U\Lambda \bar V^T =0.$ 
Thus
$\nabla^2J(aW_c,a^{-1}S_c)[(G_{i,j}, \vzero)]
= (u_i\ve_j^T a^{-2} \Lambda^2, \vzero)
=  a^{-2}\lambda_j^2 (G_{i,j} ,\vzero),$ and
$\nabla^2J(aW_c,a^{-1}S_c)[(\vzero, H_{j,i})]
= (\vzero,  a^2 \ve_j v_i^T)
=  a^2 (\vzero, H_{j,i} ).$
\end{proof}

Lemma \ref{lem:g1} and Lemma \ref{lem:g2}  
have identified $(m+n)k -2k^2$ eigenvectors.
The remaining $2k^2$ eigenvectors are found by 
considering indices for which 
$u_i$ is a column of $\bar  U$.

For each column $\bar u_j$ of $\bar U,$ $j\in \Int 1k,$  there exists $i \in \Int 1m$
such that $\bar u_j = u_i,$ $\bar v_j = v_i,$ and $\lambda_j = \sigma_i.$
If $i\in \Int 1r,$ then $\lambda_j = \sigma_i>0;$ otherwise $i\in \Int {r+1} m$ and
$\lambda_j =\sigma_i=0$. We can partition the index set $\Int 1k$ accordingly into 
$\cS,$ with $s\in \cS$ if $\lambda_s>0$, and $\cT,$ with $t\in \cT$ if $\lambda_t=0.$
We assume $\cS$ and $\cT$ are nonempty. But the case when one of $\cS$ or $\cT$ is empty, is also covered by the result below.

\begin{lemma}\label{lem:g3}
(a) For each $j\in \Int 1 k$ and $s\in \cS,$ there exist $\beta_{j,s}, \beta'_{j,s} \in \R$ 
with $\beta_{j,s}\beta'_{j,s}= -1,$ 
such that $(\bar u_j \ve_s^T, \beta_{j,s} \ve_j \bar v_s^T)$ and 
$(\bar u_j \ve_s^T, \beta'_{j,s} \ve_j \bar v_s^T)$ 
are eigenvectors of $\nabla^2 J(W_ca,a^{-1}S_c)$ 
with corresponding eigenvalues $\rho_{j,s} = 0$ and 
$\rho'_{j,s} = {\lambda_s^2}/{a^2} + a^2 > 0,$ respectively.

(b) For each $j\in \Int 1 k$ and $t\in \cT,$
$(\bar u_j \ve_t, \vzero),$ and $(\vzero, \ve_j \bar v_t^T),$ 
are eigenvectors of $\nabla^2J(aW_c,a^{-1}S_c)$ with eigenvalues 
$0$ and $\rho_0=a^2>0,$ respectively.
\end{lemma}
\begin{proof}
Recall that $W_c=\bar U,$ $S_c= \Lambda \bar V^T,$ $W_c^TW_c = I_k,$ 
and $S_cS_c^T=\Lambda^2.$ 
In addition, $\bar u_j^T X = \lambda_j \bar v_j,$ $X\bar v_j = \lambda_j u_j,$ 
$\bar u_j^T (W_cS_c-X) =\vzero,$ and $(W_cS_c-X) \bar v_j =\vzero,$ $j\in \Int 1k$. 

(a) Consider $(G, H) = (\bar u_j \ve_s^T, \beta \ve_j \bar v_s^T),$ 
for $ j\in \Int 1k$, $s\in \cS,$ and $\beta \neq 0.$
By \eqref{eq:hessianmap}, $(G, H)$ is an eigenvector of 
$\nabla^2J(W_ca,a^{-1}S_c)$ with eigenvalue $\rho$ if and only if
\begin{align*}
&\bar u_j \ve_s^T (a^{-2} \Lambda^2) 
	+ \beta \bar U \ve_j \bar v_s^T \bar V \Lambda
	+ \beta (\bar U \Lambda \bar V^T -X )\bar v_s \ve_j^T
	=  (a^{-2}\lambda_s^2 + \beta \lambda_s ) \bar u_j \ve_s^T
	= \rho \bar u_j \ve_s^T, \\
&\beta a^2 \ve_j \bar v_s^T 
	+ \bar U^T \bar u_j \ve_s^T \Lambda \bar V^T 
	+ \ve_s \bar u_j^T(\bar U \Lambda \bar V - X)
	=  (\beta a^2 + \lambda_s ) \ve_j \bar v_s^T
	= \rho \ve_j \bar v_s^T .
\end{align*}
Solving these equations for $\rho$ gives
$ \rho = a^{-2}\lambda_s^2 + \beta \lambda_s 
= a^2 + \lambda_s/\beta.$
These are the equations in \eqref{eq:rho_alpha_eqs} except that 
$\sigma_i$ has been replaced by $-\lambda_s$ and $\alpha$ by $\beta$. 
After these adjustments to \eqref{eq:quad_alpha} we find two real roots 
$\beta^+$ and $\beta^{-}$ with,
\begin{equation}\textstyle
\beta^\pm 
= -\frac{1}{2\lambda_s} \left ( \frac{\lambda_s^2 -a^4}{a^2} \pm 
\sqrt{\left ( \frac{\lambda_s^2 - a^4}{a^2}\right )^2 + 4 \lambda_s^2}~ \right ).
	\label{eq:beta_pm}
\end{equation}
Note that $s\in \cS$ and hence $\lambda_s>0$.
Let $d^\prime$ denote the square root term in \eqref{eq:beta_pm}.
Then 
$
d^\prime = \frac{(\lambda_s^4 -2 a^4 \lambda_s^2 + a^8 
	+ 4\lambda_s^2 a^4)^{\nhalf} }{a^2}
	=  \frac{\lambda_s^2 + a^4}{a^2}.
$
Hence
$\beta^+ = - \nicefrac{\lambda_s}{a^2},$ 
and $\beta^-  = \nicefrac{a^2}{\lambda_s}.$
Note that $\beta^+\beta^- = -1.$
The expressions for $\beta^\pm$ yield the following eigenvalues 
and corresponding eigenvectors
\begin{align}
\rho(\beta^+) &= 0, &&\textrm{eigenvector: } 
	(\bar u_j\ve_s^T, -\nicefrac{\lambda_s}{a^2} \ve_j \bar v_s^T);\\
\rho(\beta^-) &= {\lambda_s^2}/{a^2} + a^2>0,
&&\textrm{eigenvector: }
	(\bar u_j\ve_s^T, - \nicefrac{a^2}{\lambda_s} \ve_j \bar v_s^T).
\end{align}

(b) For $(\bar u_j \ve_t, \vzero),$ and 
$(\vzero, \ve_j \bar v_t^T),$
$j\in \Int 1 k,$ $t\in \cT,$ we have 
$\nabla^2 J(W_ca,a^{-1} S_c) [(\bar U_j \ve_t,\vzero)]  = (\bar u_j \ve_t^T (a^{-2} \Lambda^2), \bar U^T \bar u_j \ve_t^T \Lambda \bar V^T)
= (\bar u_j \ve_t^T (a^{-2} \lambda_t^2), \ve_j \lambda_t \bar v_t^T).$ Noting that $\lambda_t=0$, this simplifies to $0 (\bar u_j \ve_t^T , \vzero).$
Similarly,
$\nabla^2 J(W_ca,a^{-1} S_c) [(\vzero, \ve_j \bar v_t^T)] =
( \bar U \ve_j \bar v_t^T \bar V \Lambda , a^2 \ve_j \bar v_t^T)
= (\bar u_j \ve_t^T \lambda_t , a^2 \ve_j \bar v_t^T)
= a^2 (\vzero, \ve_j v_t^T).$
Thus $\{ (\bar u_j \ve_t, \vzero)\colon j\in \Int 1 k, t\in \cT \}$ is a set of $k|\cT|$ 
eigenvectors of  $\nabla^2 J(W_ca,a^{-1} S_c)$ with eigenvalue $0$,  and 
$\{ (\vzero , \ve_j \bar v_t^T )\colon j\in \Int 1 k, t\in \cT \}$ is a set of $k|\cT|$ 
eigenvectors with eigenvalue $a^2>0$.
\end{proof}

The above three lemmas 
have displayed $k(n+m)$ eigenvalues of $\nabla^2J(W_ca, a^{-1} S_c).$
As expected, as $a$ varies over the nonzero reals, 
the positive eigenvalues remain positive and the negative eigenvalues remain negative. 
We are particularly interested in the negative eigenvalues.

\begin{lemma}\label{lem:eigorder}
$\lambda_{\min}(\nabla^2 J(W_ca,a^{-1}S_c))<0$ if and only if
there exists $p\in \Int 1 k$ with $\lambda_p <\sigma_p$.
\end{lemma}
\begin{proof}
(Only If) 
Assume $\nabla^2 J(W_ca,a^{-1}S_c)$ has a negative eigenvalue. 
Then by Lemma \ref{lem:g2} there exists $i\in \Int 1 r$ such that 
$u_i$ not a column in $\bar U,$ and $j\in \Int 1k$ such that  $\lambda_j <\sigma_i.$ 
For some $s\in \Int 1m,$ $\bar u_j = u_s$ and hence $\lambda_j = \sigma_s < \sigma_i$. Thus $i < s$. So $u_i$ has been omitted from $\bar U$ and $\sigma_i$ from the diagonal of $\Lambda$. Yet $\sigma_s <\sigma_i$ is included in the diagonal of $\Lambda$.
It follows that the diagonal of $\Lambda$ does not contain a set of $k$ largest singular values of $X$. Hence there is a least $p\in \Int 1k$ such that $\lambda_p <\sigma_p$. \\
(If) Assume that for some $j\in \Int 1 k$, $\lambda_j < \sigma_j$.
Then there exists a least $p$ with $\lambda_p < \sigma_p$. 
Thus for $j\in \Int 1 {p-1}$, $\lambda_j = \sigma_j$ and a corresponding 
left singular vector $u_j$ occupies column $j$ of $\bar U$.  But $\bar u_p$ is
a left singular vector for a singular value $\lambda_p <\sigma_p$. 
Hence a corresponding left singular vector $u_p$ for $\sigma_p$ 
has been omitted from $\bar U$. Thus there exists a
left singular vector $u_p$,  such that $u_p$ is not a column of 
$\bar U =\vzero$, and an integer $j=p\in \Int 1k,$ such that 
$\lambda _j < \sigma_p$.
\end{proof}

By Lemma \ref{lem:eigorder}, $(W_c,S_c)$ is a strict saddle if and only if it is 
\missatp.
We now determine the least eigenvalue of 
$\nabla^2J(W_ca, a^{-1} S_c)$ 
when $(W_c,S_c)$ is a strict saddle.

\begin{theorem}\label{thm:H_mineig}
Assume $(W_c,S_c)$ is \missatp\ and $p$ is the least integer
for which $\lambda_p < \sigma_p$. Then 
\begin{align}\label{eq:lambdaminq=k}
\lambda_{\min} (\nabla^2J(W_ca, a^{-1} S_c)) 
&=\textstyle
\frac{1}{2} \left (\frac{\lambda_k^2 +a^4}{a^2} 
		- \sqrt{\left(\frac{\lambda_k^2 - a^4}{a^2}\right)^2 + 4\sigma_p^2} 
		\right ) <0
\end{align}
\end{theorem}
\begin{proof}
Since there exists $j\in \Int 1 k$ with $\lambda_j < \sigma_j$,
Lemma \ref{lem:eigorder} implies $\nabla^2 J(Wa,a^{-1}S)$ has a negative eigenvalue $\rho$, and Lemma \ref{lem:g2} implies
\begin{equation}\label{eq:mag_rho}
\textstyle
2 |\rho| = \sqrt{\left(\frac{\lambda^2 - a^4}{a^2}\right)^2 + 4\sigma^2}
		 - \frac{\lambda^2 +a^4}{a^2}
\end{equation}
where $\lambda \in \{\lambda_j \colon j\in \Int 1 k\}$, 
$\sigma \in \{\sigma_i\colon \sigma_i >\lambda_j, u_i^T \bar U=\vzero\}$.
We have
$$ \textstyle
2 \frac{d|\rho|}{ d \lambda} = \frac{(\frac{2\lambda}{a^2})(\frac{\lambda^2+a^4}{a^2}) }{ \sqrt{\left(\frac{\lambda^2 - a^4}{a^2}\right)^2 + 4\sigma^2}} - \frac{2\lambda}{a^2}
= \frac{2\lambda}{a^2} \left(
\frac{(\frac{\lambda^2+a^4}{a^2}) }{ \sqrt{\left(\frac{\lambda^2 - a^4}{a^2}\right)^2 + 4\sigma^2}} - 1 \right) = \frac{2\lambda}{a^2} \left (\frac{c}{d} -1 \right ).
$$
The terms $c$ and $d$, defined in \eqref{eq:def_cd}, satisfy $c<d$. 
Hence $\frac{d|\rho|}{ d \lambda}$ is negative
and $|\rho|$ is monotone decreasing in $\lambda$.  
It is clear from \eqref{eq:mag_rho} that $|\rho|$ is monotonically
increasing in $\sigma$. 
So we let $p$ be the least index such that $\sigma_p >\lambda_p$. 
Then $\sigma_p>\lambda_p \geq \lambda_k$.
So we select $\lambda =\lambda_k.$ 
This gives \eqref{eq:lambdaminq=k}.
\end{proof}

\begin{corollary}\label{cor:T2_lambda_min}
Under the assumptions of Theorem \ref{thm:H_mineig}, 
\begin{equation}\label{eq:H_eigmin_q=k_2}
\textstyle
\lambda_{\min} (\nabla^2J(W_ca, a^{-1} S_c)) 
= \frac{-(\sigma_p^2 - \lambda_k^2)}{ \left(\frac{\lambda_k^2}{2a^2} + \frac{a^2}{2}\right) +\sqrt{\sigma_p^2 +\left (\frac{\lambda_k^2}{2a^2} - \frac{a^2}{2} \right)^2} }
\end{equation}
\end{corollary}
\begin{proof}
By Theorem \ref{thm:H_mineig}, 
$\lambda_{\min} (\nabla^2J(W_c a, a^{-1} S_c)) $ is given by \eqref{eq:lambdaminq=k}.
Let $c=\frac{\lambda_k^2 + a^4}{a^2}$
and
$d= \sqrt{\left(\frac{\lambda_k^2 - a^4}{a^2}\right)^2 + 4\sigma_p^2}.$
Then $\lambda_{\min} (\nabla^2J(Wa, a^{-1} S))  = \frac{1}{2} (c-d)$.
Use $c^2-d^2 = (c-d)(c+d)$ to write 
$\lambda_{\min} (\nabla^2J(Wa, a^{-1} S)) 
= \frac{1}{2} \frac{c^2-d^2 }{c+d}.$ 
Noting that $c^2-d^2 = -4(\sigma_p^2 - \lambda_k^2)$ 
and evaluating $1/2(c+d)$ using the definitions of $c$ and $d$ yields   \eqref{eq:H_eigmin_q=k_2}.
\end{proof}

\subsection{A Canonical Point with $\mathbf{q<k}$} \label{sec:cp_q<k}

We now determine the eigenvalues and eigenvectors of the Hessian at a
canonical point $(W_c,S_c)$ with
$W_c = \begin{bmatrix} W& \vzero_{k-q}\end{bmatrix},$
$S_c =\begin{bmatrix} S^T & V_0C_0 \end{bmatrix}^T,$
$k-q>0,$ and $\rank(W)=q.$  
Noting that $(W, S)$ is a full rank canonical point for dimension $q,$
we first ``lift'' the $q(m+n)$ eigenvalues and eigenvectors 
of $\nabla^2 J(W,S)$ to eigenvalues and 
eigenvectors of $\nabla^2 J(W_c,S_c).$

\begin{lemma} \label{lemma:relate-evals}
If $(G,H)$ is an eigenvector of  $\nabla^2J(W,S)$ with eigenvalue $\rho$
then $(\begin{bmatrix}G & \vzero_{k-q} \end{bmatrix},$ $\begin{bmatrix} H^T & \vzero_{k-q}\end{bmatrix}^T)$ is an eigenvector of $\nabla^2J(W_c,S_c)$ with eigenvalue $\rho$.
\end{lemma}
\begin{proof}
Simple algebra shows that  $(W_c, S_c)$, 
$(\begin{bmatrix}G & \vzero_{k-q} \end{bmatrix},
\begin{bmatrix} H^T & \vzero_{k-q}\end{bmatrix}^T)$ and $\rho$
	satisfy \eqref{eq:Heigvec_0}. 
\end{proof}

We obtain $(k-q)(n-m)$ additional eigenvalues from
Lemma \ref{lem:g1}. Specifically,
for $i\in \Int {m+1} n$ and $j \in \Int {q+1}k,$ 
$(\vzero, H_{j,i})$ is an eigenvector of $\nabla^2J(W_c,S_c)$ with eigenvalue $\rho_0=1$. 

There are $2(k-q)m - 2 (k-q)^2$ remaining eigenvalues. 
Since $C_0^TC_0 \in \R^{(k-q)\by (k-q)}$ is symmetric positive semidefinite, $C_0^TC_0=Z\Omega Z^T$ where $Z \in \OG {k-q}$,
and $\Omega$ is diagonal with the eigenvalues $\omega_j$ of $C_0^TC_0$  
listed in decreasing order on the diagonal.
Let $\tilde z_j = [\vzero, z_j^T ]^T.$ 
Then for $i\in \Int 1 m$ with $u_i$ not in $\bar U,$ and $j\in \Int {q+1} {k},$
set $\GC_{i,j} = u_i \tilde z_j^T$  and $\HC_{j,i} = \tilde z_j v_i^T.$ 

\begin{lemma}\label{lem:qkC0}
For $i\in \Int 1 m$ with $u_i$ not a column in $\bar U,$ 
and $j\in \Int {q+1} {k},$ the following hold:

(a) If $i\in \Int 1 r,$ there exist $\delta_{i,j}, \delta'_{i,j} \in \R$ with 
$\delta_{i,j}\delta'_{i,j}= -1$ such that $(\GC_{i,j}, \delta_{i,j} \HC_{j,i})$ 
and $(\GC_{i,j}, \delta'_{i,j} \HC_{j,i})$ are eigenvectors of 
$\nabla^2 J(W_c, S_c)$ with eigenvalues
$\rho_{i,j}$ and $\rho'_{i,j}$ given in \eqref{eq:rho+-}.

(b) If $i\in \Int {r+1} m,$ the pair of vectors $(\GC_{i,j},  \vzero)$ and 
$(\vzero, \HC_{j,i})$,  are eigenvectors of $\nabla^2J(W_c,S_c)$ 
with eigenvalues  $\omega_j \geq 0,$ and $0,$ respectively.
\end{lemma}

\begin{proof}
(a) Write $W=\bar U$ and $S=\Lambda \bar V^T.$ 
Then $W_cS_c = WS= \bar U \Lambda \bar V^T.$ In addition,  
$W_c^TW_c = \left [\begin{smallmatrix} I_q & \vzero\\ \vzero & \vzero\end{smallmatrix}\right ]$ and using Lemma \ref{lem:S2props} part (a), 
$S_cS_c^T = \left [\begin{smallmatrix} \Lambda^2 & \vzero\\ \vzero & C_0^TC_0\end{smallmatrix}\right ].$ 
From \eqref{eq:hessianmap}, the first component of $\nabla^2 J(W_c,S_c)[(\GC_{i,j},\delta \HC_{j,i})]$ is
$
u_i \left [ \vzero \ z_j^T \right] \left [\begin{smallmatrix}\Lambda^2 & \vzero\\ \vzero & Z\Omega Z^T\end{smallmatrix}\right ] + \delta \left [ \bar U \ \vzero \right ]
\left [ \begin{smallmatrix}  \vzero \\ z_j \end{smallmatrix} \right ] v_i^T S_c^T +
\delta (\bar U \Lambda \bar V^T - X) v_i \tilde z_j^T
= (\omega_j - \delta \sigma_i) u_i \tilde z_j^T,$
and the second component is
$\delta \left [ \begin{smallmatrix}  I_q & \vzero\\ \vzero & \vzero \end{smallmatrix}\right] 
\left [ \begin{smallmatrix} \vzero \\ z_j\end{smallmatrix} \right ] v_i^T
+ \left[ \begin{smallmatrix} \bar U^T \\ \vzero\end{smallmatrix}\right ] u_i \tilde z_i^T S_c
+ \tilde z_j u_i^T (\bar U \Lambda \bar V^T - X) = -\sigma_i \bar z_j v_i^T.$
We obtain an eigenvector with eigenvalue $\rho$ if and only if 
$\rho = (\omega_j - \delta \sigma_i) = -\frac{\sigma_i}{\delta}.$
This equation appears in the proof of Theorem \ref{thm:ev0c0} 
and the analysis there yields the result.

(b) The first component of $\nabla^2 J(W_c,S_c)[(\GC_{i,j},\vzero)]$ is
$u_i \left [ \vzero \ \tilde z_j^T \right] \left [\begin{smallmatrix}\Lambda^2 & \vzero\\ \vzero & Z\Omega Z^T\end{smallmatrix}\right ] 
= \omega_j u_i \tilde z_j^T;$
the second is
$\left[ \begin{smallmatrix} \bar U^T \\ \vzero\end{smallmatrix}\right ] u_i \tilde z_i^T S_c
+ \tilde z_j u_i^T (\bar U \Lambda \bar V^T - X) = -\sigma_i \bar z_j v_i^T =\vzero.$
Thus $\nabla^2 J(W_c,S_c)[(\GC_{i,j},\vzero)] = \omega_j (\GC_{i,j},\vzero).$
The first component of $\nabla^2 J(W_c,S_c)[(\vzero, \HC_{i,j})]$ is
$\left[ \bar U \ \vzero\right ] \left [ \begin{smallmatrix} \vzero \\ z_j^T \end{smallmatrix}\right ]
v_i^T S_c^T + (\bar U \Lambda \bar V^T - X)v_i \tilde z_j^T= \vzero;$
the second is
$\left[ \begin{smallmatrix} I_q & \vzero \\ \vzero & \vzero \end{smallmatrix}\right ] 
\left [ \begin{smallmatrix} \vzero \\ z_j^T \end{smallmatrix}\right ] v_i^T
=
\vzero.$ 
Thus $\nabla^2 J(W_c,S_c)[(\vzero, \HC_{i,j})] = 0 (\vzero, \HC_{i,j}).$
\end{proof}

It remains to find $2(k-q)^2$ eigenvalues. 
To do so, consider indices $i\in \Int 1 m$ for which $u_i$ 
is a column of $\bar  U$.
For each column $\bar u_j$ of $\bar U,$ 
$j\in \Int 1q,$  there exists $i \in \Int 1m$
such that $\bar u_j = u_i,$ $\bar v_j = v_i,$ and $\lambda_j = \sigma_i.$ 
We partition the index set $\Int 1q$ into $\cS,$ with $s\in \cS$ 
if $\lambda_s>0$, and $\cT,$ with $t\in \cT$ if $\lambda_t=0.$ 
One can verify that the remaining eigenvalues and eigenvectors of the Hessian have the form given in Lemma \ref{lem:g3} 
for when $k=q$ and $a=1$.
This determines all the eigenvalues.

To find the least eigenvalue of $\nabla^2 J(W_c,S_c)$ we consider two situations. 
First, using the notation defined above, if $(W,S) = (\bar U, \Lambda \bar V^T)$ is \nomiss, the eigenvalues of $\nabla^2 J(W,S)$ are non-negative. Hence the only negative eigenvalues of $\nabla^2 J(W_c,S_c)$ are given in Lemma \ref{lem:qkC0}. 
The least among these is clearly 
\begin{align}
\label{eq:lam_min_q<k}
\lambda_{\min}(\nabla^2 J(W_c,S_c)) = 
\textstyle \frac{\omega_{k-q}}{2} - \sqrt{ \sigma_{q+1}^2 +\left(\frac{\omega_{k-q}}{2}\right)^2}.
\end{align}
Second, if $(W,S)$ is \missatp, then there is a least $p\in \Int 1 q$ with 
$\sigma_p>\lambda_p$. In this case, negative eigenvalues emerge from two sources. First, by ``lifting'' the eigenvectors of $\nabla^2J(W,S)$ (Lemma \ref{lemma:relate-evals} and Theorem \ref{thm:H_mineig}). Second, negative eigenvalues arise because of the existence of zero columns in $W_c$ (Lemma \ref{lem:qkC0}).
The least eigenvalue of $\nabla^2J(W,S)$
is then the minimum eigenvalue from these two sources. 
Using Lemma \ref{lem:qkC0} and Theorem \ref{thm:H_mineig}, 
this can be expressed as 
\begin{align} 
\begin{split}\label{eq:lam_min_q<k_not}
\lambda_{\min}(\nabla^2 J(W_c,S_c)) 
&=  \min \Big \{ \textstyle \frac{\omega_{k-q}}{2} - \sqrt{\sigma_{p}^2 +\left(\frac{\omega_{k-q}}{2}\right)^2}, \\
&\qquad\qquad \shalf \left (\lambda_q^2 +1 
		- \sqrt{\left(\lambda_q^2 - 1\right)^2 + 4\sigma_p^2} 
		\right )
 \Big \}.
 \end{split}
\end{align}
The expression above simplifies when $C_{0} = \vzero.$ 

\begin{theorem}\label{thm:lammin_C0=0}
For a canonical point $(W_c,S_c)$ with $q<k$ and $C_0=\vzero$, 
$$
\lambda_{\min}(\nabla^2J(W_c,S_c)) =
\begin{cases}
-\sigma_{q+1}, & \textrm{if } (W_c,S_c) \textrm{ is \nomiss};\\
-\sigma_p, & \textrm{if } (W_c,S_c) \textrm{ is \missatp};
\end{cases}
$$
where in the second case, $p \in \Int 1 q$ is the least index with 
$\lambda_p < \sigma_p.$
\end{theorem}
\begin{proof}
When $C_0 = \vzero$, $\omega_j=0$ for $j \in \Int 1{k-q}$. When $(W_c,S_c)$ is \nomiss, using \eqref{eq:lam_min_q<k}, the least eigenvalue is $-\sigma_{q+1}$.
Otherwise, there are two possibilities. We show that $-\sigma_p$ is the least eigenvalue. Let
$
-\sigma_p < \textstyle \frac{\lambda_q^2 +1}{2} 
		- \sqrt{\left(\frac{\lambda_q^2 - 1}{2}\right)^2 + \sigma_p^2}.
$
This is equivalent to 
$\sigma_p +  \frac{\lambda_q^2 +1}{2} >  \sqrt{\left(\frac{\lambda_q^2 - 1}{2}\right)^2 + \sigma_p^2}.$ Squaring both sides, expanding, and eliminating common terms confirms the claimed ordering.
\end{proof}

\section{The Manifold $\mathbf{W^T W - S S^T = C}$ } \label{sec:append-intersection}

\begin{lemma} [{\bf \cite[Theorem 1]{AroraICML2018a} }]
\label{lem:invariance}
Along every solution of the gradient flow o.d.e.,  $W_t^TW_t -S_tS_t^T$ is a constant symmetric $k\by k$ matrix and
$\|W_t\|_F^2 - \|S_t\|_F^2$ is a constant.
\end{lemma}

\begin{proof} 
Given an initial condition $(W_0, S_0 ) \in \SP,$
the gradient flow o.d.e. $\frac{d}{dt} {(W_t,S_t)} = -  \nabla J(W_t,S_t)$
defines a curve $(W_t, S_t)$, $t\geq 0$, in $\SP$. 
Taking the inner product of both sides of the o.d.e
with $T_{W,S}(H)$ in \eqref{eq:glg_tang}, and using 
$\ip{T_{W,S}(H) }{\nabla J(W,S)} =0$, yields
$H^TW_t^T \dot W_t - S_t^TH^T \dot S_t = \ip{H}{W_t^T\dot W_t - \dot S_t S_t^T} =\vzero$, $\forall H \in \R^{k \times k}$. Thus 
$W_t^T\dot W_t - \dot S_t S_t^T = \vzero$.
Adding this to its transpose gives
$\frac{d}{dt} (W_t^TW_t -S_tS_t^T) = \vzero$. Thus
$W_t^TW_t-S_tS_t^T = W_0^TW_0 - S_0S_0^T$. Now note that 
$\|W_t\|_F^2  -\|S_t\|_F^2 = \trace(W_t^T W_t) - \trace(S_t^TS_t) = 
 \trace(W_t^T W_t - S_tS_t^T)$, and $\trace(W_t^T W_t - S_tS_t^T)$ is a constant. 
\end{proof}

\bibliographystyle{unsrt}
\bibliography{references}

\end{document}